\documentclass[10pt,final,twocolumn]{IEEEtran}
\long\def\comment#1{}

\usepackage{textcomp}
\usepackage{cite}
\usepackage{graphicx}
\usepackage{subfigure}
\usepackage{psfrag}
\usepackage{url}
\usepackage{stfloats}
\usepackage{amsmath}
\usepackage{amssymb,amsfonts}
\usepackage{amsthm}
\usepackage{float}
\usepackage[usenames]{color}
\usepackage{algorithm,algorithmic}

\newtheorem{corollary}{Corollary}
\newtheorem{assumption}{Assumption}
\newtheorem{remark}{Remark}

\newtheorem{lemma}{Lemma}
\newtheorem{theorem}{Theorem}
\newtheorem{example}{Example}

\begin{document}

\setlength{\arraycolsep}{0.3em}

\title{DOT and DOP: Linearly Convergent Algorithms for Finding Fixed Points of Multi-Agent Operators
\thanks{}}

\author{Xiuxian Li, {\em Member, IEEE}, Min Meng, and Lihua Xie, {\em Fellow IEEE}
\thanks{This research was supported by Ministry of Education, Singapore, under grant AcRF TIER 1- 2019-T1-001-088 (RG72/19), the National Natural Science Foundation of China under Grant 62003243, Shanghai Municipal Commission of Science and Technology No. 19511132101, Shanghai Municipal Science and Technology Major Project under no. 2021SHZDZX0100. {\em(Corresponding author: Lihua Xie.)}}
\thanks{X. Li and M. Meng are with Department of Control Science and Engineering, College of Electronics and Information Engineering, Institute for Advanced Study, and Shanghai Research Institute for Intelligent Autonomous Systems, Tongji University, Shanghai, China (e-mail: xli@tongji.edu.cn, mengmin@tongji.edu.cn).}
\thanks{L. Xie is with School of Electrical and Electronic Engineering, Nanyang Technological University, 50 Nanyang Avenue, Singapore (e-mail: elhxie@ntu.edu.sg).}
%\thanks{Y. Hong is with Key Laboratory of Systems and Control, Institute of Systems Science, Chinese Academy of Sciences, Beijing, 100190, China (email: yghong@iss.ac.cn).}
}

\maketitle

\setcounter{equation}{0}
\setcounter{figure}{0}
\setcounter{table}{0}
%%%%%%%%%%%%%%%%%%%%%%%%%%%%%%%%%%%%%%%%%%%%%%%%%%%%%%%%%%%%%%%%%%%%%%%%%%%%%%%%%%%%%%%%%

\begin{abstract}
This paper investigates the distributed fixed point finding problem for a global operator over a directed and unbalanced multi-agent network, where the global operator is quasi-nonexpansive and only partially accessible to each individual agent. Two cases are addressed, that is, the global operator is sum separable and block separable. For this first case, the global operator is the sum of local operators, which are assumed to be Lipschitz, and each local operator is privately known to each individual agent. To deal with this scenario, a distributed (or decentralized) algorithm, called Distributed quasi-averaged Operator Tracking algorithm (DOT), is proposed and rigorously analyzed, and it is shown that the algorithm can converge to a fixed point of the global operator at a linear rate under a linear regularity condition, which is strictly weaker than the strong convexity assumption on cost functions in existing convex optimization literature. For the second scenario, the global operator is composed of a group of local block operators which are Lipschitz and can be accessed only by each individual agent. In this setup, a distributed algorithm, called Distributed quasi-averaged Operator Playing algorithm (DOP), is developed and shown to be linearly convergent to a fixed point of the global operator under the linear regularity condition. The above studied problems provide a unified framework for many interesting problems. As examples, the proposed DOT and DOP are exploited to deal with distributed optimization and multi-player games under partial-decision information. Finally, numerical examples are presented to corroborate the theoretical results.
\end{abstract}


\section{Introduction}\label{s1}

Fixed point theory in real Hilbert spaces is known as a powerful tool in a variety of domains such as optimization, engineering, economics, game theory, and nonlinear numerical analysis \cite{bauschke2017convex,cegielski2012iterative}. Generally speaking, the main goal is to devise algorithms for computing a fixed point of an operator.

Up to now, plenty of research has extensively addressed centralized algorithms for finding a fixed point of nonexpansive or quasi-nonexpansive operators in the literature \cite{liang2016convergence,borwein2017convergence,bravo2018rates,themelis2019supermann}, where a central/global coordinator or computing unit is able to access all information of the studied problem. It is known that the typical Picard iteration in general does not converge for nonexpansive operators (e.g., a simple example is the operator $-Id$ with nonzero initial points, where $Id$ is the identity operator), although it usually performs well for contractive operators. For nonexpansive operators, one prominent algorithm is the so-called Krasnosel'ski\u{\i}-Mann (KM) iteration \cite{mann1953mean,krasnosel1955two}, and it is shown to converge weakly to a fixed point of a nonexpansive operator in real Hilbert spaces under mild conditions \cite{reich1979weak}.

In recent few decades, distributed (or decentralized) algorithms have been an active topic in a wide range of domains, including fixed point theory, computer science, game theory, and control theory, and so on, mostly inspired by the fact that distributed algorithms, in contrast with centralized ones, possess a host of fascinating advantages, such as low cost, robustness to failures or antagonistic attacks, privacy preservation, low computational complexity, and so on. Distributed algorithms do not assume global/central coordinators or computing units and instead a finite group of agents (e.g., computing units and robots, and so on.), who may be spatially separated, aim to solve a global problem in a collaborative manner by local information exchanges. Wherein, local information exchanges are often depicted by a simple graph, connoting that every agent can interact with only a subset of agents, instead of all agents. Along this line, distributed algorithms have thus far been investigated extensively under both fixed and time-varying communication graphs in distributed optimization \cite{nedic2009distributed,liu2017convergence,li2019distributed}, game theory \cite{mejia2019solutions}, and multi-agent systems/networks \cite{ren2008distributed,li2018quantized}, to just name a few.

In more recent years, distributed algorithms have received a growing attention in the fixed point finding problem \cite{fullmer2018distributed,fullmer2016asynchronous,liu2017distributed,alaviani2019distributed,li2019distributed3,li2019distributed5}. For instance, a synchronous distributed algorithm was proposed in \cite{fullmer2018distributed} for computing a common fixed point of a collection of paracontraction operators, and for the same problem, an asynchronous distributed algorithm was developed in \cite{fullmer2016asynchronous}. Meanwhile, different from paracontraction operators, another type of operators (i.e., strongly quasi-nonexpansive operators) was addressed for the common fixed point seeking problem in \cite{liu2017distributed} by designing a distributed algorithm in the presence of time-varying delays under the assumption that the communication graph is repeatedly jointly strongly connected. It should be noted that many interesting problems can boil down to the common fixed point finding problem, such as convex feasibility problems \cite{necoara2018randomized,kruger2018set} and the problem of solving linear algebraic equations in a distributed fashion \cite{mou2015distributed,wang2019distributed,alaviani2018distributed}, and so forth. For example, the linear algebraic equation solving problem was formulated such that the distribution of random communication graphs is not required, which include asynchronous updates and/or unreliable interconnection protocols. Notice that all the aforementioned works are in the Euclidean space. Regarding the Hilbert space, the authors in \cite{alaviani2019distributed} investigated distributed optimization under random and directed interconnection graphs, where a distributed algorithm was proposed and shown to be convergent in both almost surely and mean square sense along with the introduction of a novel convex minimization problem over the fixed-value point set of a nonexpansive random operator. In addition, the authors in \cite{li2019distributed3} took into account the common fixed point finding problem for a finite collection of nonexpansive operators, where two distributed algorithms were proposed with a full coordinate updating and a random block-coordinate updating, respectively, and compared with \cite{fullmer2018distributed,fullmer2016asynchronous,liu2017distributed}, the contributions of \cite{li2019distributed3} lie in the study of real Hilbert spaces, the consideration of operator errors, and the establishment of a sublinear convergence speed. Furthermore, a more general scenario, where no common fixed points are assumed for all local operators, was investigated in \cite{li2019distributed5}, where two distributed algorithms were devised to resolve the problem. It is noteworthy that, to our best knowledge, \cite{li2019distributed5} is the first to investigate the fixed point finding problem of a global operator in real Hilbert spaces, where the global operator is an average of local operators over a multi-agent network. Nevertheless, the convergence rate is not analyzed for the proposed algorithms in \cite{li2019distributed5}.

Motivated by the above facts, the purpose of this paper is to further investigate the fixed point finding problem of a quasi-nonexpansive global operator over a time-invariant, directed, and unbalanced communication graph. Two scenarios are taken into account, that is, the global operator is sum separable and block separable. In the first case, the global operator is composed of a sum of local operators, and in the second case, the global operator is comprised of a family of local block operators. That is, in both cases the global operator is separable and consists of local operators, which are assumed to be Lipschitz and are only privately accessible to each individual agent, thereby needing all agents to tackle the global problem in a collaborative manner. The contributions of this paper are threefold as follows.
\begin{enumerate}
  \item For the first case, a distributed algorithm, called \underline{d}istributed quasi-averaged \underline{o}perator \underline{t}racking algorithm (DOT), is developed and shown to be convergent to a fixed point of the global operator at a linear rate under a linear regularity condition. Compared with the closely related work \cite{li2019distributed5}, where no convergence speed is provided, a different algorithm is developed here and shown to be convergent at a linear rate. It should be noted that the problem here is more general than the common fixed point seeking problem \cite{fullmer2016asynchronous,fullmer2018distributed,liu2017distributed,li2019distributed3}, where all local operators are assumed to have at least one common fixed point, while this assumption is dropped here. As a special case, linear convergence can also be ensured for the common fixed point seeking problem. In contrast, \cite{li2019distributed3} only provides a sublinear rate for nonexpansive operators.
  \item For the second case, a distributed algorithm, called \underline{d}istributed quasi-averaged \underline{o}perator \underline{p}laying algorithm (DOP), is proposed, which is shown to be linearly convergent to a fixed point of the global operator under the linear regularity condition. To our best knowledge, this is the first to study the block separable case in a decentralized manner.
  \item The studied setups in this paper provide a unified framework for a host of interesting problems. For example, the proposed DOT and DOP algorithms can be leveraged to resolve distributed optimization and multi-player games under partial-decision information.
\end{enumerate}

A preliminary version of the paper was presented at a conference \cite{li2020dot}. The present paper extends the result of \cite{li2020dot} in several ways. \cite{li2020dot} only considers the sum separable case without providing detailed proofs of the main result. In comparison, a full proof of the main result (i.e., Theorem \ref{t1}) for the sum separable case is provided. Besides, one more scenario is investigated in this paper, i.e., the block separable case (see Theorem \ref{t3}). Also, one more applications is presented here, i.e., multi-player games under partial-decision information, along with one more numerical example.

The structure of this paper is as follows. Some basic knowledge and the problem formulation are introduced in Section II, and the first case with the global operator being sum separable is addressed in Section III, followed by the second case with the global operator being block-coordinate separable in Section IV. Several applications are provided in Section V. Some numerical examples are presented in Section VI, and the conclusion is drawn in Section VII.

\section{Preliminaries and Problem Formulation}\label{s2}

\subsection{Notations}\label{s2.1}

Let $\mathcal{H}$ be a real Hilbert space with inner product $\langle\cdot,\cdot\rangle$ and associated norm $\|\cdot\|$. Define $[N]:=\{1,2,\ldots,N\}$ for any integer $N>0$, and denote by $col(z_1,\ldots,z_k)$ the column vector or matrix by stacking up $z_i,i\in [k]$. Given an integer $n>0$, denote by $\mathbb{R}$, $\mathbb{R}^n$, $\mathbb{R}^{n\times n}$, and $\mathbb{N}$ the sets of real numbers, $n$-dimensional real vectors, $n\times n$ real matrices, and nonnegative integers, respectively. Let $P_X(\cdot)$ represent the projection operator onto a closed and convex set $X\subseteq\mathcal{H}$, i.e., $P_X(z):=\mathop{\arg\min}_{x\in X}\|z-x\|$ for $z\in\mathcal{H}$. Moreover, denote by $I$, $Id$, and $\otimes$ the identity matrix of appropriate dimension, the identity operator, and the Kronecker product, respectively. Let ${\bf 1}_n$ be an $n$-dimensional vector with all entries $1$ for an integer $n>0$, and the subscript is omitted when the dimension is clear in the context. $d_X(z):=\inf_{x\in X}\|z-x\|$ denotes the distance from $z\in\mathcal{H}$ to the set $X$. For an operator $T:\mathcal{H}\to\mathcal{H}$, define $Fix(T):=\{x\in\mathcal{H}|T(x)=x\}$ to be the set of fixed points of $T$ and $T_\beta:=Id+\beta(T-Id)$, called a {\em $\beta$-relaxation} of $T$ with a relaxation parameter $\beta\geq 0$. Denote by $M_\infty$ the infinite power of a square matrix $M$, i.e., $M_\infty=\lim_{k\to\infty} M^k$, if it exists, and let $\rho(M)$ and $det(M)$ be the spectral radius and determinant of $M$, respectively.

\subsection{Operator Theory}\label{s2.2}

Consider an operator $T:S\to\mathcal{H}$ for a nonempty set $S\subseteq\mathcal{H}$. $T$ is called {\em $L$-Lipschitz (continuous)} for a constant $L>0$ if
\begin{align}
\|T(x)-T(y)\|\leq L\|x-y\|,~~~\forall x,y\in S.         \label{1}
\end{align}
Further, $T$ is called {\em nonexpansive} (resp. {\em contractive}) if $L=1$ (resp. $L<1$), {\em quasi-nonexpansive (QNE)} if (\ref{1}) holds with $L=1$ for all $x\in S$ and $y\in Fix(T)$, and {\em $\rho$-strongly quasi-nonexpansive ($\rho$-SQNE)} for $\rho>0$ if it holds that for all $x\in S$ and $y\in Fix(T)$,
\begin{align}
\|T(x)-y\|^2\leq \|x-y\|^2-\rho\|x-T(x)\|^2.               \label{5}
\end{align}

$T$ is called {\em $\eta$-averaged} (resp. {\em $\eta$-quasi-averaged}) for $\eta\in (0,1)$ if it can be written as
\begin{align}
T=(1-\eta)Id+\eta R,         \label{2}
\end{align}
where $R$ is some nonexpansive (resp. quasi-nonexpansive) operator.

The aforementioned concepts can be found in Section 4 of \cite{bauschke2017convex}, where quasi-averaged operators are defined here as an analogue of averaged operators. It is well known that when $T$ is QNE, the fixed point set $Fix(T)$ is closed and convex (cf., Corollary 4.24 in \cite{bauschke2017convex}).

The operator $T$ is said {\em boundedly linearly regular} if for any bounded set $\mathbb{K}\subseteq S$, there exists a constant $\omega>0$ such that
\begin{align}
d_{Fix(T)}(x)\leq \omega \|x-Tx\|,~~~\forall x\in \mathbb{K}         \nonumber
\end{align}
and $T$ is called {\em linearly regular} if $\omega$ is independent of $\mathbb{K}$ \cite{bauschke1996projection}.

It is easy to observe that the (bounded) linear regularity means that the distance between $x$ and $Tx$ is lower bounded by a scaled distance from the vector $x$ to the set $Fix(T)$, and bounded linear regularity is weaker than linear regularity. For instance, the projection operator $P_C$ is linearly regular with constant $1$, where $C\subseteq\mathcal{H}$ is nonempty, closed and convex, and when $\mathcal{H}=\mathbb{R}$, the thresholder operator
\begin{align}
T(x)=\left\{
       \begin{array}{ll}
         0, & \text{if~}|x|\leq 1 \\
         x-1, & \text{if~}x>1 \\
         x+1, & \text{if~}x<-1
       \end{array}
     \right.               \nonumber
\end{align}
is boundedly linear regular with $\omega=\max\{|x|,1\}$, but not linearly regular \cite{bauschke2015linear}. Moreover, it has been shown in \cite{banjac2018tight} that linear regularity is necessary and sufficient for global Q-linear convergence to a fixed point for an SQNE operator. More details on (bounded) linear regularity can be found in \cite{bauschke2015linear,banjac2018tight,cegielski2018regular}.

\subsection{Problem Formulation}\label{s2.3}

The aim of this paper is to compute a fixed point of a {\em global operator} $F:\mathcal{H}\to\mathcal{H}$, i.e.,
\begin{align}
\text{find}~x\in\mathcal{H},~~\text{s.t.}~~x\in Fix(F).            \label{6C}
\end{align}

It is worth mentioning that problem (\ref{6C}) in real Hilbert spaces (but not in $\mathbb{R}^n$) can find applications, for example, in digital signal processing and $L^2([0,\pi])$ (i.e., the space of all square integrable functions $f:[0,\pi]\to \mathbb{R}$) \cite{debnath2005introduction}, and so on.

In this paper, no global/central coordinator, master, or computing unit is assumed to exist for problem (\ref{6C}), and however the global operator $F$ is separable and consists of local operators, which are privately accessible to each individual agent in a network. Briefly speaking, only partial information of $F$ can be privately known by each agent in a network, which is interesting and realistic in large-scale problems, as extensively studied in distributed optimization and distributed machine learning, and so on. Specifically, we focus on two scenarios: 1) $F$ is sum separable and 2) $F$ is block separable, as elaborated below.

{\em Case 1: Sum Separable.} $F$ is {\em sum separable} over a network of $N$ agents, that is,
\begin{align}
\text{{\bf (Problem I)}~~~~~find}~x\in Fix(F),~~F=\frac{1}{N}\sum_{i=1}^N F_i,            \label{6}
\end{align}
where each $F_i:\mathcal{H}\to\mathcal{H}$ is a {\em local operator}, only privately accessible to agent $i$ for $i\in[N]$. Note that the formulation (\ref{6}) is also investigated in \cite{li2019distributed5}, where, however, no convergence rates are provided, while a linear convergence speed is established in this paper. And the formulation (6) is more general than the common fixed point finding problem \cite{fullmer2018distributed,fullmer2016asynchronous,liu2017distributed,alaviani2019distributed,li2019distributed3,li2019distributed5} and the linear algebraic equation solving problem \cite{mou2015distributed,wang2019distributed,alaviani2018distributed}, as discussed in the introduction section.

{\em Case 2: Block (or Block-Coordinate) Separable.} In this case, $\mathcal{H}=\mathcal{H}_1\oplus\cdots\oplus\mathcal{H}_N$ is the direct Hilbert sum, where every $\mathcal{H}_i,i\in[N]$ is a real Hilbert space, along with the same inner product $\langle\cdot,\cdot\rangle$ and associated norm $\|\cdot\|$ as $\mathcal{H}$, i.e., $\mathcal{H}$ and $\mathcal{H}_1$ are consistent when $N=1$. Let
$x=(x_1,\ldots,x_N)$ denote a generic vector in $\mathcal{H}$ with $x_i\in\mathcal{H}_i,i\in[N]$. Then the global operator $F$ can be written as a block-coordinate version $F=(\texttt{F}_1,\ldots,\texttt{F}_N)$, where $\texttt{F}_i:\mathcal{H}\to\mathcal{H}_i$ for $i\in[N]$, i.e., $F(x)=(\texttt{F}_1(x),\ldots,\texttt{F}_N(x))$ for $x\in\mathcal{H}$. In this setup, $F$ is {\em block (or block-coordinate) separable} over a network of $N$ agents, i.e.,
\begin{align}
\text{{\bf (Problem II)}~~~~~find}~x\in Fix(F),~F=(\texttt{F}_1,\ldots,\texttt{F}_N),            \label{6b}
\end{align}
where each $\texttt{F}_i$ is a {\em local operator}, only privately accessible to agent $i$ for all $i\in[N]$. Also, each agent $i$ only knows its own vector $x_i$ with no knowledge of $x_j$'s for all $j\neq i$. To our best knowledge, this scenario is novel and also practical. For example, in multi-player games, it is difficult or impossible for a global/central coordinator or master to know the whole coordinates of $x$ due to privacy.

With the above discussion, the objective of this paper is to develop distributed (or decentralized) algorithms to solve problems I and II in real Hilbert spaces.

\begin{remark}
To illustrate applications of the above studied problems, a simple example in function approximation is provided here, which is useful such as in reinforcement learning (cf., Chapter 9 in \cite{sutton2018reinforcement}). Consider a (reward) function $\texttt{r}:\mathbb{R}^n\to\mathbb{R}$, which may be unknown in reality, and let us approximate it by $\sum_{j=1}^\infty w_j \exp(-\frac{\|x-c_j\|^2}{2\sigma_j^2})$ based on radial basis functions (RBFs), where $c_j$ and $\sigma_j$ are some prespecified parameters (e.g., feature's center state and feature's width in reinforcement learning, respectively), $w=(w_1,w_2,\ldots)\in\ell^2$ is the variable to be optimized, and $\ell^2$ is the space of square-summable sequences which is an infinite-dimensional Hilbert space. Then the performance is to minimize the approximation error $f(w):=\frac{1}{N}\sum_{i=1}^N f_i(w)$ with $f_i(w):=\frac{1}{N_i}\sum_{i=1}^{N_i}|\texttt{r}(s_i)-\sum_{j=1}^\infty w_j \exp(-\frac{\|s_i-c_j\|^2}{2\sigma_j^2})|^2$, where $\{s_i\}_{i=1}^{N_i}$ are a set of sample data privately known by agent $i$. It is easy to verify that $f_i$ is differentiable and convex with $L_i$-Lipschitz gradient for some constant $L_i>0$, thus implying that the operator $F_i:w\mapsto w-\xi\nabla f_i(w)$ is nonexpansive for a constant $\xi\in(0,2/L)$ with $L:=\max_{i\in[N]}L_i$ (Lemma 4 in \cite{liang2016convergence}). Then the problem is equivalent to the fixed point finding problem (\ref{6}) with $\mathcal{H}=\ell^2$. More applications will be given in Section \ref{sec-app}.
\end{remark}

\subsection{Graph Theory}\label{s2.4}

The communication pattern among all agents is captured by a simple directed graph, denoted by $\mathcal{G}=(\mathcal{V},\mathcal{E})$, where $\mathcal{V}=[N]$ and $\mathcal{E}\subseteq\mathcal{V}\times\mathcal{V}$ are the node (or agent) and edge sets, respectively. An {\em edge} $(i,j)\in\mathcal{E}$ means that agent $i$ can send information to agent $j$, but not vice versa, and agent $i$ (resp. $j$) is called an {\em in-neighbor} or simply {\em neighbor} (resp. out-neighbor) of agent $j$ (resp. $i$). A graph is called undirected if and only if $(i,j)\in\mathcal{E}$ amounts to $(j,i)\in\mathcal{E}$, and directed otherwise. A {\em directed path} is defined to be a sequence of adjacent edges $(i_1,i_2),(i_2,i_3),\ldots,(i_{l-1},i_{l})$, and a graph is said {\em strongly connected} if any two nodes can be connected by a directed path from one to the other.

\subsection{Assumptions}

With the above preparations, we are now ready to impose some standard assumptions.

\begin{assumption}[Strong Connectivity]\label{a2}
The communication graph $\mathcal{G}$ is strongly connected. Moreover,
\begin{enumerate}
  \item two matrices $A=(a_{ij})\in\mathbb{R}^{N\times N}$ (row-stochastic) and $B=(b_{ij})\in\mathbb{R}^{N\times N}$ (column-stochastic) are arbitrarily assigned to $\mathcal{G}$ with $a_{ii}>0,b_{ii}>0$ for all $i\in[N]$;
  \item denote by the left stochastic eigenvector (resp. right stochastic eigenvector) of $A$ (resp. $B$) associated with the eigenvalue $1$ as $\pi=col(\pi_1,\ldots,\pi_N)$ (resp. $\nu=col(\nu_1,\ldots,\nu_N)$) such that $A_\infty={\bf 1}_N \pi^\top$ (resp. $B_\infty=\nu{\bf 1}_N^\top$) and $\pi_i>0$, $\nu_i>0$ for all $i\in[N]$.
\end{enumerate}
\end{assumption}

It is worth mentioning that the directed graph is not required to be balanced in this paper. Note that, similar to \cite{xin2019distributed}, $A$ and $B$ are consistent with $\mathcal{G}$, that is, $a_{ij}>0$ and $b_{ij}>0$ if and only if $(j,i)\in \mathcal{E}$ for $i\neq j$. Notice that $A$ and $B$ do not need to be doubly stochastic. Additionally, the property in Assumption \ref{a2}.2 can be ensured by the strong connectivity of $\mathcal{G}$ \cite{xin2019distributed}.

\begin{assumption}\label{a0}
$F_i$ is Lipschitz with constant $L_i$ for all $i\in[N]$, i.e., $\|F_i(x)-F_i(y)\|\leq L_i\|x-y\|,~\forall x,y\in\mathcal{H}$.
\end{assumption}

\begin{assumption}\label{a1}
~
\begin{enumerate}
\item $F$ is quasi-nonexpansive with $Fix(F)\neq\emptyset$.
\item $F$ is linearly regular, i.e., there exists a constant $\kappa>0$ such that
\begin{align}
d_{Fix(F)}(x)\leq \kappa \|F(x)-x\|,~~~\forall x\in \mathbb{H}.                \label{FA}
\end{align}
\end{enumerate}
\end{assumption}

\begin{remark}\label{r1}
Note that Assumption \ref{a1} is only made for the global operator $F$, and is not necessary for any local operators $F_i$'s. In addition, the linear regularity is strictly weaker than the strong convexity of functions (Section III in \cite{banjac2018tight}), where the linear regularity of functions is for operators involving functions' gradients, as shown in Section \ref{app-1}.
\end{remark}

\section{The DOT Algorithm for Problem I}\label{s3}

This section aims to develop a distributed algorithm for tackling problem (\ref{6}) which can converge at a linear rate. Without loss of generality, the vectors in $\mathcal{H}$ are viewed as column vectors in this section.

For problem (\ref{6}), if $F$ can be known by a global/central computing unit (or coordinator), then a famous centralized algorithm, called the KM iteration \cite{matsushita2017convergence}, can be exploited, i.e.,
\begin{align}
x_{k+1}=x_k+\alpha_k (F(x_k)-x_k),        \label{8}
\end{align}
where $\{\alpha_k\}_{k\in\mathbb{N}}$ is a sequence of relaxation parameters with $\alpha_k\in[0,1]$. Note that the KM iteration usually applies to nonexpansive operators, but it still works for quasi-nonexpansive operators here under the linear regularity condition in Assumption \ref{a1}. However, the centralized iteration (\ref{8}) is not realistic here since no global/central computing unit (or coordinator) exists in our setting, which hence motivates us to devise distributed (or decentralized) algorithms based on only local information exchanges among all agents.

Motivated by the classical KM iteration and the tracking techniques such as those in \cite{zhang2019decentralized,xin2019distributed,li2018distributedon}, a \underline{d}istributed quasi-averaged \underline{o}perator \underline{t}racking algorithm (DOT) is proposed as
\begin{subequations}
\begin{align}
x_{i,k+1}&=\sum_{j=1}^N a_{ij}x_{j,k}+\alpha\Big(\frac{y_{i,k}}{w_{i,k}}-\sum_{j=1}^N a_{ij}x_{j,k}\Big),              \label{9a}\\
y_{i,k+1}&=\sum_{j=1}^N b_{ij}y_{j,k}+F_i(x_{i,k+1})-F_i(x_{i,k}),                                         \label{9b}\\
w_{i,k+1}&=\sum_{j=1}^N b_{ij}w_{j,k},                                                               \label{9c}
\end{align}               \label{9}
\end{subequations}
\hspace{-0.18cm} where $x_{i,k}$ is an estimate of a fixed point of the global operator $F$ by agent $i$ at time $k\geq 0$ for all $i\in[N]$, and $\alpha\in(0,1)$ is the stepsize to be determined. Set the initial conditions as: arbitrary $x_{i,0}\in\mathcal{H}$, $y_{i,0}=F_i(x_{i,0})$, and $w_{i,0}=1$ for all $i\in[N]$. It is noteworthy that neighboring agents are only involved in (\ref{9}) for each agent due to $a_{ij}=0$ and $b_{ij}=0$ when agent $j$ is not a neighbor of agent $i$.

Roughly speaking, $y_{i,k}$ is employed to track the weighted global operator $\nu_i\sum_{i=1}^N F_i(x_{i,k})$, and meanwhile $w_{i,k}$ is a scalar used to track $\nu_i N$ in order to counteract the imbalance of the matrix $B$ in (\ref{9b}).

For (\ref{9c}), it is easy to verify that each $w_{i,k}$ will exponentially converge to $N\nu_i$. However, invoking the method in \cite{charalambous2015distributed}, it can be concluded that the final value $N\nu_i$ can be evaluated for each agent $i$ in finite time in a distributed manner. Because of this, without loss of generality, algorithm (\ref{9}) can be rewritten by replacing $w_{i,k}$ with $N\nu_i$ as in Algorithm 1.

To facilitate the following analysis, Algorithm 1 can be written in a compact form
\begin{align}
x_{k+1}&=\mathbf{A}x_k+\alpha\Big[\frac{1}{N}(D_\nu^{-1}\otimes Id)y_k-\mathbf{A}x_k\Big],               \label{cf1}\\
y_{k+1}&=\mathbf{B}y_k+\mathbf{F}(x_{k+1})-\mathbf{F}(x_k),                                            \label{cf2}
\end{align}
where $x_k,y_k$ are concatenated vectors of $x_{i,k},y_{i,k}$, respectively, $\mathbf{A}:=A\otimes Id$, $\mathbf{B}:=B\otimes Id$, $D_\nu:=diag\{\nu_1,\ldots,\nu_N\}$, and $\mathbf{F}(z):=col(F_1(z_1),\ldots,F_N(z_N))$ for a vector $z=col(z_1,\ldots,z_N)\in\mathcal{H}^N:=\mathcal{H}\times\cdots\times\mathcal{H}$ (the $N$-fold Cartesian product of $\mathcal{H}$).

\begin{algorithm}
 \caption{\underline{D}istributed Quasi-Averaged \underline{O}perator \underline{T}racking (DOT)}
 \begin{algorithmic}[1]
% \renewcommand{\algorithmicrequire}{\textbf{Require:}}
% \renewcommand{\algorithmicensure}{\textbf{Output:}}
% \REQUIRE
%% \ENSURE  out
%%\\  \textit{Initialization}:
  \STATE \textbf{Initialization:} Stepsize $\alpha$ in (\ref{cdc13}), communication matrices $A$ and $B$, and local initial conditions $x_{i,0}\in\mathcal{H}$ and $y_{i,0}=F_i(x_{i,0})$ for all $i\in[N]$.
  \STATE \textbf{Iterations:} Step $k\geq 0$: update for each $i\in[N]$:
\begin{subequations}
\begin{align}
x_{i,k+1}&=\sum_{j=1}^N a_{ij}x_{j,k}+\alpha\Big(\frac{y_{i,k}}{N\nu_i}-\sum_{j=1}^N a_{ij}x_{j,k}\Big),              \label{A1-1}\\
y_{i,k+1}&=\sum_{j=1}^N b_{ij}y_{j,k}+F_i(x_{i,k+1})-F_i(x_{i,k}).                                         \label{A1-2}
\end{align}             \label{A1}
\end{subequations}
 \end{algorithmic}
\end{algorithm}

To proceed, it is helpful to introduce two new weighted norms for the Cartesian product $\mathcal{H}^N$, i.e.,
\begin{align}
\|z\|_\pi:=\sqrt{\sum_{i=1}^N \pi_i\|z_i\|^2},~~~\|z\|_\nu:=\sqrt{\sum_{i=1}^N \frac{\|z_i\|^2}{\nu_i}}           \label{cdc3}
\end{align}
for any vector $z=col(z_1,\ldots,z_N)\in\mathcal{H}^N$. Let $\|\cdot\|$ be the natural norm in $\mathcal{H}^N$, i.e., $\|z\|:=\sqrt{\sum_{i=1}^N \|z_i\|^2}$. Additionally, it is also necessary to introduce two weighted norms in $\mathbb{R}^N$ \cite{xin2019distributed}, i.e., for any $x=col(x_1,\ldots,x_N)\in\mathbb{R}^N$,
\begin{align}
\|x\|_\pi:=\sqrt{\sum_{i=1}^N \pi_i x_i^2},~~~\|x\|_\nu:=\sqrt{\sum_{i=1}^N \frac{x_i^2}{\nu_i}}.              \label{cdc4}
\end{align}
Please note that the notations $\|\cdot\|_\pi$ and $\|\cdot\|_\nu$ in (\ref{cdc3}) and (\ref{cdc4}) should be easily distinguished by the context. Accordingly, let us denote by $\|M\|_\pi$ and $\|M\|_\nu$ (resp. $\|M\otimes T\|_\pi$ and $\|M\otimes T\|_\nu$) the norms for a matrix $M\in\mathbb{R}^{N\times N}$ (resp. a matrix $M\in\mathbb{R}^{N\times N}$ and an operator $T$ in $\mathcal{H}$) induced by $\|\cdot\|_\pi$ and $\|\cdot\|_\nu$ in (\ref{cdc4}) (resp. (\ref{cdc3})), respectively.

It is easy to see that the natural norm $\|\cdot\|$ is equivalent to $\|\cdot\|_\pi$, $\|\cdot\|_\nu$ in (\ref{cdc3}), (\ref{cdc4}), and thus to the induced matrix norms, that is, there are positive constants $c_i,i\in[4]$ such that
\begin{align}
c_1\|\cdot\|&\leq \|\cdot\|_\pi\leq c_2\|\cdot\|,                   \label{eq1}\\
c_3\|\cdot\|_\nu&\leq \|\cdot\|_\pi\leq c_4\|\cdot\|_\nu.           \label{eq2}
\end{align}

Then the following results can be obtained.
\begin{lemma}[\cite{xin2019distributed}]\label{l2}
For all $x\in\mathbb{R}^N$, there hold
\begin{align}
&\|Ax-A_\infty x\|_\pi\leq \rho_1\|x-A_\infty x\|_\pi,            \label{cdc5}\\
&\|Bx-B_\infty x\|_\nu\leq \rho_2\|x-B_\infty x\|_\nu,            \label{cdc6}\\
&\|A\|_\pi=\|A_\infty\|_\pi=\|I_N-A_\infty\|_\pi=1,               \label{cdc7}\\
&\|B\|_\nu=\|B_\infty\|_\nu=\|I_N-B_\infty\|_\nu=1,               \label{cdc8}
\end{align}
where $\rho_1:=\|A-A_\infty\|_\pi<1$ and $\rho_2:=\|B-B_\infty\|_\nu<1$.
\end{lemma}

\begin{lemma}\label{l3}
For all $z\in\mathcal{H}^N$, the following statements hold
\begin{align}
\|\mathbf{A}z-\mathbf{A}_\infty z\|_\pi &\leq \rho_1\|z-\mathbf{A}_\infty z\|_\pi,        \label{cdc9-1}\\
\|\mathbf{B}z-\mathbf{B}_\infty z\|_\nu &\leq \rho_2\|z-\mathbf{B}_\infty z\|_\nu,        \label{cdc9-2}\\
\|I_N\otimes Id-\mathbf{A}_\infty\|_\pi &=\|I_N-A_\infty\|_\pi=1,                       \label{cdc9-3}
\end{align}
where $\mathbf{A}_\infty:=A_\infty\otimes Id$ and $\mathbf{B}_\infty:=B_\infty\otimes Id$.
\end{lemma}
\begin{proof}
The proof can be found in Appendix A.
\end{proof}

\begin{lemma}[\cite{horn2012matrix}]\label{l4}
For an irreducible nonnegative matrix $M\in\mathbb{R}^{n\times n}$, it is primitive if it has at least one non-zero diagonal entry.
\end{lemma}

\begin{lemma}[\cite{horn2012matrix}]\label{l5}
For an irreducible nonnegative matrix $M\in\mathbb{R}^{n\times n}$, there hold (i) $\rho(M)>0$ is an eigenvalue of $M$, (ii) $Mx=\rho(M)x$ for some positive vector $x$, and (iii) $\rho(M)$ is an algebraically simple eigenvalue.
\end{lemma}

\begin{lemma}[\cite{li2020distributed}]\label{l-m}
For $X,Y\in\mathbb{R}^{n\times n}$, let $\lambda$ be a simple eigenvalue of $X$. Denote $u$ and $v$ respectively the left and right eigenvectors of $X$ corresponding to $\lambda$. Then, it holds that
\begin{enumerate}
  \item for each $\epsilon >0$, there exists a $\delta>0$ such that, $\forall t\in\mathbb{C}$ with $|t|<\delta$, there is a unique eigenvalue $\lambda(t)$ of $X+tY$ such that $|\lambda(t)-\lambda-t\frac{u^\top Y v}{u^\top v}|\leq |t|\epsilon$,
  \item $\lambda(t)$ is continuous at $t=0$, and $\lim_{t\to 0}\lambda(t)=\lambda$,
  \item $\lambda(t)$ is differentiable at $t=0$, and $\frac{d\lambda(t)}{dt}\big|_{t=0}=\frac{u^\top Y v}{u^\top v}$.
\end{enumerate}
\end{lemma}

\begin{lemma}\label{l6}
It holds that $\bar{y}_k=\sum_{i=1}^N F_i(x_{i,k})$, where $\bar{y}_k:=\sum_{i=1}^N y_{i,k}$.
\end{lemma}
\begin{proof}
Left multiplying (\ref{cf2}) by ${\bf 1}^\top$ yields that $\bar{y}_{k+1}=\bar{y}_k+\sum_{i=1}^N F_i(x_{i,k+1})-\sum_{i=1}^N F_i(x_{i,k})$, which further implies that $\bar{y}_k-\sum_{i=1}^N F_i(x_{i,k})=\bar{y}_0-\sum_{i=1}^N F_i(x_{i,0})$. Note that $y_{i,0}=F_i(x_{i,0})$. The conclusion directly follows.
\end{proof}

To move forward, an important result for the convergence analysis is first given below.

%\begin{proposition}\label{prop1}
%Consider the $L$-quasi-Lipschitz operator $F$ with $Fix(F)\neq\emptyset$. For any set $\emptyset\neq\mathbb{K}\subseteq\mathcal{H}$, $d_{Fix(F)}(F_\alpha(x))\leq \vartheta d_{Fix(F)}(x)$ with some $\vartheta\in [0,1)$ for all $x\in\mathbb{K}$ if and only if $F_\alpha$ is linearly regular on $\mathbb{K}$, i.e., $d_{Fix(F_\alpha)}(x)\leq \kappa\|F_\alpha(x)-x\|$ with some $\kappa>0$ for all $x\in\mathbb{K}$, under the condition $\alpha\in(0,\min\{1,\alpha_B\})$, where $\alpha_B$ is the positive root of equation $\alpha^2\kappa^2(\delta^2-1)+\alpha-1=0$. Moreover, it holds
%\begin{align}
%\vartheta=\sqrt{1+\alpha(\delta^2-1)-\frac{1-\alpha}{\alpha\kappa^2}}\in[0,1).            \nonumber
%\end{align}
%\end{proposition}
%\begin{proof}
%The proof is similar to that of Proposition \ref{prop2}, thus omitted here.
%\end{proof}
%
%Note that Proposition \ref{prop1} is more general than Theorem 1 in \cite{banjac2018tight}, which cannot be applied to the cases where $\delta>1$.
%
%Proposition \ref{prop1} directly implies the following pivotal result.

\begin{lemma}\label{l1}
Under Assumption \ref{a1}, if $\alpha\in(0,1-\delta]$, where $\delta\in (0,1)$ is any pre-specified parameter, then there holds
\begin{align}
d_{Fix(F)}(F_\alpha(x))\leq\rho_3 d_{Fix(F)}(x),~~~\forall x\in\mathbb{H}               \label{cdc1}
\end{align}
where $F_\alpha:=Id+\alpha(F-Id)$ is the $\alpha$-quasi-averaged operator of $F$, and
\begin{align}
\rho_3:=1-\frac{\delta\alpha}{4\kappa^2}\in[0,1).        \label{cdc2}
\end{align}
\end{lemma}

\begin{proof}
It is easy to see that $F_\alpha-Id=\alpha(F-Id)$, which together with (\ref{FA}) yields that $d_{Fix(F)}(x)\leq \kappa\|F(x)-x\|=\frac{\kappa}{\alpha}\|F_\alpha(x)-x\|$ for all $x\in\mathbb{H}$. Therefore, $F_\alpha$ is linearly regular with constant $\frac{\kappa}{\alpha}$. Simultaneously, it is known that each $\alpha$-quasi-averaged operator is $\frac{1-\alpha}{\alpha}$-SQNE \cite{banjac2018tight}, and thus $F_\alpha$ is $\frac{1-\alpha}{\alpha}$-SQNE. With the above two properties of $F_\alpha$ as well as $Fix(F_\alpha)=Fix(F)$, invoking Theorem 1 in \cite{banjac2018tight} leads to $d_{Fix(F)}(F_\alpha(x))\leq \phi d_{Fix(F)}(x)$, where $\phi:=\sqrt{1-\frac{\alpha(1-\alpha)}{\kappa^2}}\in[0,1)$. Meanwhile, it is easy to verify that
\begin{align}
\phi\leq 1-\frac{\alpha(1-\alpha)}{2\kappa^2}\leq 1-\frac{\delta\alpha}{4\kappa^2},                  \nonumber
\end{align}
where $\alpha\leq 1-\delta$ is used in the last inequality. This ends the proof.
\end{proof}

We are now ready to give the main result of this section.
\begin{theorem}\label{t1}
Under Assumptions \ref{a2}-\ref{a1}, all $x_{i,k}$'s generated by Algorithm 1 converge to a common point in $Fix(F)$ at a linear rate, if there holds
\begin{align}
0<\alpha<\min\{1-\delta,\alpha_c\},                \label{cdc13}
\end{align}
where $\alpha_c$ is the smallest positive real root of equation $det(I-M(\alpha))=0$, and
\begin{align}
M(\alpha):=\left(
             \begin{array}{ccc}
               (1-\alpha)\rho_1 & \alpha c_2\theta_1 & 0 \\
               \theta_2(\alpha\theta_3+\theta_4) & \rho_2+\alpha\theta_1\theta_2 & 2\alpha c_1\theta_2 \\
               \frac{\alpha \bar{L}}{c_1} & \frac{\alpha\sqrt{N}\theta_1}{c_1} & 1-\frac{\delta\alpha}{4\kappa^2} \\
             \end{array}
           \right)                 \label{cdc14}
\end{align}
with $\bar{L}:=\max_{i\in[N]}\{L_i\}$, $\theta_1:=\frac{c_4\|D_\nu^{-1}\|}{N}$, $\theta_2:=\frac{c_2\bar{L}(\sqrt{N}+1)}{c_1c_3}$, $\theta_3:=\rho_1+\bar{L}$, and $\theta_4:=\|A-I\|$.
\end{theorem}
\begin{proof}
The proof can be found in Appendix B.
\end{proof}

\begin{remark}\label{rm2}
It should be noticed that the problem considered in this paper is more general than the common fixed point finding problem in \cite{fullmer2016asynchronous,fullmer2018distributed,liu2017distributed,li2019distributed3}, where all local operators are assumed to have at least one common fixed point, while this is dropped in this paper. It is worthwhile to notice the linear algebraic equation solving problem in \cite{mou2015distributed,wang2019distributed,alaviani2018distributed} can be cast as a special case of the common fixed point seeking problem. Note that no convergence speeds are provided in \cite{fullmer2016asynchronous,fullmer2018distributed,liu2017distributed,alaviani2019distributed}, although random interconnection graphs are considered in \cite{alaviani2019distributed}. In addition, the same problem as here is also studied in \cite{li2019distributed5} for nonexpansive operators, where the convergence rate is not analyzed, while a linear convergence rate is established here and more general operators are considered, i.e., quasi-nonexpansive operators. It should be also noted that a main difference between DOT here and D-KM in \cite{li2019distributed5} is that DOT exploits a tracking technique for $F$ with a constant stepsize, similar to the tracking idea for a global gradient in distributed optimization \cite{zhang2019decentralized,xin2019distributed,li2018distributedon}, while D-KM does not use this idea and applies a diminishing stepsize.
\end{remark}

As a special case, when all local operators have at least one common fixed point, problem (\ref{6}) will reduce to the common fixed point seeking problem due to $Fix(F)=\cap_{i=1}^N Fix(F_i)$ in this case (e.g., Proposition 4.47 in \cite{bauschke2017convex}). Therefore, we have the following result.
\begin{corollary}\label{c1}
Under the same conditions in Theorem \ref{t1}, if all $F_i$'s have at least one common fixed point, then all $x_{i,k}$'s generated by Algorithm 1 converge to a common point in $\cap_{i=1}^N Fix(F_i)$ at a linear rate.
\end{corollary}

Note that the convergence speed is also analyzed for the common fixed point seeking problem in \cite{li2019distributed3} (i.e., the DO algorithm), where the main difference between DOT and DO is that an estimate is introduced for each agent to track the global operator $F$ here, while each agent does not perform this tracking in the DO algorithm. However, the rate is sublinear and all operators are assumed to be nonexpansive in \cite{li2019distributed3}, while a linear rate is provided here in Corollary \ref{c1} and less conservative operators, i.e., quasi-nonexpansive operators, are considered. Note that time-varying communication graphs were considered with non-identical stepsizes for the DO algorithm in \cite{li2019distributed3}, while this paper is concerned with static communication graphs with an identical stepsize for all agents. Along this line, it is interesting to further address the case with non-identical stepsizes for different agents and time-varying communication graphs in future.

\section{The DOP Algorithm for Problem II}\label{s-II}

This section is concerned with solving problem (\ref{6b}). Without loss of generality, the vectors in $\mathcal{H}$ are viewed as row vectors in this section for convenient analysis.

For problem (\ref{6b}), each agent $i\in[N]$ can only privately access $\texttt{F}_i$ with its own vector $x_i$ for a whole vector $x=(x_1,\ldots,x_N)\in\mathcal{H}$ over a network of $N$ agents, where $x_i$ is privately known by agent $i$ itself, as commonly encountered in multi-player games, and so on. To handle this problem, each agent $i\in[N]$ maintains a vector $x_k^i=(x_{1,k}^i,\ldots,x_{N,k}^i)\in\mathcal{H}$ at time step $k\geq 0$ as an estimate of a fixed point of $F$, where $x_{j,k}^i$ is an estimate of $x_{j,k}$ (i.e., the vector of agent $j$ at time $k$) by agent $i$ at time $k$ with $x_{i,k}^i=x_{i,k}$. That is, each agent $i$ updates its own vector $x_{i,k}$ at time slot $k$ without access to the vectors of all other agents $j\neq i$, and thus each agent $i$ needs to estimate other agents' vectors $x_{j,k}$ denoted as $x_k^i$ at each time $k\geq 0$ over the communication graph $\mathcal{G}$ satisfying Assumption \ref{a2}.

Now, a distributed algorithm is proposed as in Algorithm 2, where $A=(a_{ij})\in\mathbb{R}^{N\times N}$ is the communication matrix introduced after Assumption \ref{a2}, which is only row-stochastic, and $x_{-i,k}^j:=(x_{1,k}^j,\ldots,x_{i-1,k}^j,x_{i+1,k}^j,\ldots,x_{N,k}^j)$ for all $i,j\in[N]$, i.e., $x_{-i,k}^j$ is the agent $j$'s estimate of all agents' vectors except the $i$-th agent. We recall that $\pi$ is the left stochastic eigenvector of $A$ associated with the eigenvalue $1$ as introduced in Assumption \ref{a2}.2. It should be noted that the $i$-th entry $\pi_i$ of $\pi$ can be evaluated by agent $i$ in finite time in a distributed fashion using the approach in \cite{charalambous2015distributed}. Thus, Algorithm 2 is distributed.

\begin{algorithm}
 \caption{\underline{D}istributed Quasi-Averaged \underline{O}perator \underline{P}laying (DOP)}
 \begin{algorithmic}[1]
% \renewcommand{\algorithmicrequire}{\textbf{Require:}}
% \renewcommand{\algorithmicensure}{\textbf{Output:}}
% \REQUIRE
%% \ENSURE  out
%%\\  \textit{Initialization}:
  \STATE \textbf{Initialization:} Stepsize $\alpha$ in (\ref{P7}), communication matrix $A$, and local initial conditions $x_{0}^i\in\mathcal{H}$ for all $i\in[N]$.
  \STATE \textbf{Iterations:} Step $k\geq 0$: update for each $i\in[N]$:
\begin{subequations}
\begin{align}
&x_{i,k+1}=\sum_{j=1}^N a_{ij} x_{i,k}^j+\frac{\alpha}{\pi_i}\big(\texttt{F}_i(x_k^i)-\sum_{j=1}^N a_{ij} x_{i,k}^j\big),            \label{A2-1}\\
&x_{-i,k+1}^i=\sum_{j=1}^N a_{ij}x_{-i,k}^j.                \label{A2-2}
\end{align}                                          \label{A2}
\end{subequations}
 \end{algorithmic}
\end{algorithm}

To ease the upcoming analysis, let us define $x_k:=col(x_k^1,\ldots,x_k^N)\in\mathcal{H}^N$, $\hat{x}_{i,k}:=\sum_{j=1}^N a_{ij} x_{i,k}^j$, and $\bar{F}:=diag\{(\texttt{F}_1(x_k^1)-\hat{x}_{1,k})/\pi_1,\ldots,(\texttt{F}_N(x_k^N)-\hat{x}_{N,k})/\pi_N\}$. Then algorithm (\ref{A2}) can be written in a compact form
\begin{align}
x_{k+1}=Ax_k+\alpha\bar{F}.         \label{P1}
\end{align}

Multiplying $\pi^\top$ on both sides of (\ref{P1}) yields that
\begin{align}
\tilde{x}_{k+1}=\tilde{x}_k+\alpha\tilde{F},      \label{P2}
\end{align}
where $\tilde{x}_k=(\tilde{x}_{1,k},\ldots,\tilde{x}_{N,k}):=\sum_{i=1}^N\pi_i x_k^i$ and $\tilde{F}:=(\texttt{F}_1(x_k^1)-\hat{x}_{1,k},\ldots,\texttt{F}_N(x_k^N)-\hat{x}_{N,k})$.

To move forward, it is useful to recall the weighted norm $\|z\|_\pi:=\sqrt{\sum_{i=1}^N \pi_i\|z_i\|^2}$ for a vector $z=col(z_1,\ldots,z_N)\in\mathcal{H}^N$, as defined in (\ref{cdc3}). And let $\|\cdot\|$ be a norm in $\mathcal{H}^N$ defined by $\|z\|:=\sqrt{\sum_{i=1}^N \|z_i\|^2}$. Remember that the vectors in $\mathcal{H}$ are seen as row vectors in this section. Then similar to (\ref{cdc9-1}) in Lemma \ref{l3}, it is easy to obtain the following result.
\begin{lemma}\label{l7}
For all $z\in\mathcal{H}^N$, it holds that
\begin{align}
\|Az-A_\infty z\|_\pi &\leq \rho_1\|z-A_\infty z\|_\pi,        \label{P3}
\end{align}
where $A_\infty={\bf 1}_N\pi^\top$ as defined in the paragraph after Assumption \ref{a2} and $\rho_1:=\|A-A_\infty\|_\pi<1$.
\end{lemma}

To ensure the linear convergence, Assumptions \ref{a0} and \ref{a1} are still imposed in this section, but $F_i$ in Assumption \ref{a0} is replaced with $\texttt{F}_i$, i.e., $\|\texttt{F}_i(x)-\texttt{F}_i(y)\|\leq L_i\|x-y\|,~\forall x,y\in\mathcal{H}$ for $i\in[N]$.

With the above preparations, we are now ready to give the main result of this section.
\begin{theorem}\label{t3}
Under Assumptions \ref{a2}-\ref{a1} with $F_i$ being replaced with $\texttt{F}_i$ in Assumption \ref{a0}, all $x_k^i$'s generated by DOP converge to a common point in $Fix(F)$ at a linear rate, if
\begin{align}
0<\alpha<\min\{1-\delta,\alpha_{L}\},         \label{P7}
\end{align}
where $\delta\in (0,1)$ is any pre-specified parameter, $\alpha_{L}$ is the smallest positive real root of equation $det(I-\Theta(\alpha))=0$, and
\begin{align}
\Theta(\alpha):=\left(
                  \begin{array}{cc}
                    \rho_1+\alpha \theta_5 & 2\alpha c_2\sqrt{2\varpi} \\
                    \frac{\alpha(\bar{L}+1)}{c_1} & 1-\frac{\delta\alpha}{4\kappa^2} \\
                  \end{array}
                \right)                         \label{P8}
\end{align}
with $\theta_5:=2c_2\sqrt{\varpi(\bar{L}^2+1)}/c_1$, $\bar{L}:=\max_{i\in[N]}\{L_i\}$, $\varpi:=N-1+\frac{(1-\underline{\pi})^2}{\underline{\pi}^2}$, and $\underline{\pi}:=\min_{i\in[N]}\{\pi_i\}>0$.
\end{theorem}
\begin{proof}
The proof can be found in Appendix C.
\end{proof}

\begin{remark}\label{rm5}
It is worth pointing out that the work \cite{li2019distributed5} only considers the sum separable case and does not present the convergence speed. In contrast, this paper addresses both the sum separable case (see Theorem \ref{t1}) and the block separable case (see Theorem \ref{t3}), and to our best knowledge, this paper is the first to address the block separable case for the fixed point finding problem in the decentralized fashion. Note that the block separable case in Theorem \ref{t3} has many applications as will be discussed in Section \ref{app}.
\end{remark}

\begin{remark}
Moreover, it is worth noting that Problem II can be cast as the common fixed point finding problem for a family of operators $T_i:=(Id_1,\ldots,Id_{i-1},\texttt{F}_i,Id_{i+1},\ldots,Id_N):\mathcal{H}\to \mathcal{H}$, where $Id_i:x\mapsto x_i$ for $x=(x_1,\ldots,x_N)\in\mathcal{H}$ and $i\in[N]$. However, there exist two issues: 1) the linear regularity condition may not hold for $\sum_{i=1}^N T_i$; and 2) although the DO algorithm proposed in \cite{li2019distributed3} is applicable for finding common fixed points, only a sublinear convergence speed is established.
\end{remark}

\begin{remark}\label{RM1}
It is noteworthy that in problem II each agent $i$ can only know its own vector $x_{i,k}$ at each time $k\geq 0$, but has no access to vectors $x_{j,k}$'s of all other agents for $j\neq i$. In this regard, agent $i$ needs to estimate all other $x_{j,k}$'s in order to compute the value of its operator $\texttt{F}_i$. If each agent has full access to all other agents' vectors, then a simpler algorithm can be devised to tackle this setup, i.e.,
\begin{align}
x_{i,k+1}=x_{i,k}+\alpha(\texttt{F}_i(x_k)-x_{i,k}),          \label{P16}
\end{align}
where $x_{i,k}$ is the same as in (\ref{A2}) and $x_k:=(x_{1,k},\ldots,x_{N,k})$. However, there is no need for each agent to estimate the entire vector $x_k$ in this setup. As for (\ref{P16}), a linear convergence to a fixed point of the global operator $F$ can be similarly proved to that of Theorem \ref{t3}.
\end{remark}

\section{Applications of DOT and DOP}\label{app}\label{sec-app}

The considered problems can provide a unified framework for a multitude of interesting problems. To show this, this section aims to provide two examples, i.e., distributed optimization and multi-player games under partial decision-information.

\subsection{Distributed Optimization}\label{app-1}

Consider a global optimization problem
\begin{align}
\min_{x\in\mathcal{H}}~~~f(x)             \label{P9}
\end{align}
where $f:\mathcal{H}\to\mathbb{R}$ is a differentiable and convex function, whose gradient is Lipschitz with constant $L$. It is easy to verify that problem is equivalent to finding fixed points of an operator $F:x\mapsto x-\xi\nabla f(x)$ for any given $\xi>0$, which is shown to be $(L\xi)/2$-averaged when $\xi\in(0,2/L)$ (cf., Lemma 4 in \cite{liang2016convergence}) and thus nonexpansive. For large-scale problems, the function $f$ is usually expensive or impossible to be known by a global/central coordinator or computing unit, instead it is more practical to consider the case where $f$ is separable. Along this line, two cases are discussed below.

{\em Case 1.} $f$ is sum separable, i.e., $f(x)=\frac{1}{N}\sum_{i=1}^N f_i(x)$, where $f_i:\mathcal{H}\to\mathbb{R}$ is the local function, which is differentiable and convex with $\ell_i$-Lipschitz gradient, only known to agent $i$. This problem is often called {\em distributed/decentralizd optimization}, which has been extensively studied in the literature. In this case, the problem can be equivalently cast as problem I (i.e., (\ref{6})) with $F_i: x\mapsto x-\xi\nabla f_i(x)$ for any given bounded $\xi>0$, which is Lipschitz. Therefore, Assumption \ref{a0} holds true. In this setup, DOT proposed in this paper can be leveraged to solve problem (\ref{P9}) in the sum separable case. As such, under Assumptions \ref{a2}-\ref{a1}, the linear convergence to a solution of (\ref{P9}) can be guaranteed by Theorem \ref{t1}.

{\em Case 2.} $f$ is block separable, that is, $\nabla f(x)=col(\nabla_{x_1}f(x),\ldots,\nabla_{x_N}f(x))$ with $x=col(x_1,\ldots,x_N)$, where $x_i$ is the vector of agent $i\in[N]$ and each agent $i$ is only capable of computing partial gradient $\nabla_{x_i}f(x)$ with respect to its own vector $x_i$. This scenario is realistic in some cases, partially because it is computationally expensive to compute the whole gradient $\nabla_x f(x)$ by a global/central coordinator or computing unit, and partially because only part of data $x_i$ may be privately acquired by spatially distributed agents. In this setup, the problem can be recast as problem (\ref{6b}) with $\texttt{F}_i: x\mapsto x_i-\xi\nabla_{x_i} f(x)$ for any given bounded $\xi>0$, which is Lipschitz if $\nabla_{x_i} f(x)$ is so. For this problem, under Assumptions \ref{a2}-\ref{a1}, the linear convergence to a solution of (\ref{P9}) can be ensured by Theorem \ref{t3}.

\begin{remark}\label{rm6}
Note that in Case 1, the linear convergence rate is ensured under the linear regularity in Assumption \ref{a1}, which is strictly weaker than the strong convexity of $f_i$'s or $f$ \cite{banjac2018tight}, which is widely postulated in distributed optimization \cite{jakovetic2014linear,nedic2017achieving,xu2018convergence,xin2019distributed1,pu2020push,li2020distributed}, to just name a few. Also, notice that the linear regularity is only assumed for $F$, not necessary for any local operator $F_i$. For Case 1, a similar condition to linear regularity, i.e., metric subregularity, is employed in \cite{liang2019exponential} for deriving a linear convergence, which is however in Euclidean spaces under balanced undirected communication graphs, while the result here works in a more general setting, i.e., in real Hilbert spaces under unbalanced directed graphs. It should be also noted that the aforesaid problem is just an application of a general problem (\ref{6}) addressed here. In addition, to our best knowledge, this paper is the first to investigated the Case 2 in distributed optimization.
\end{remark}

\subsection{Game Under Partial-Decision Information}\label{app-3}

Consider a noncooperative $N$-player game with unconstrained action sets, where each player can be viewed as an agent and a Nash equilibrium is assumed to exist for the game. In this problem, each player $i\in[N]$ possesses its own cost (or payoff) function $J_i(x_i,x_{-i})$, which is differentiable, where $x_i$ is the decision/action vector of player $i$ and $x_{-i}$ denotes the decision vectors of all other players, i.e., $x_{-i}:=col(x_{1,k},\ldots,x_{i-1,k},x_{i+1,k},\ldots,x_{N,k})$. Note that player $i$ cannot access other players' decision vectors, i.e., the considered game here is under partial-decision information, which is more practical than the case where each player has full access to all other players' decisions in most of existing works. For this problem, at time step $k\geq 0$, each player $i\in[N]$ will choose its own decision vector $x_{i,k}\in\mathbb{R}^{n_i}$ and a cost $J_i(x_{i,k},x_{-i,k})$ will be incurred for player $i$ after all players make their decisions. Then the objective is for each player to minimize its own cost function, that is, all players desire to achieve a {\em Nash equilibrium (NE)} $x^*=col(x_1^*,\ldots,x_N^*)\in\mathbb{R}^n$ with $n:=\sum_{i=1}^N n_i$, which is defined as: for all $i\in[N]$,
\begin{align}
J_i(x_i^*,x_{-i}^*)\leq J_i(x_i,x_{-i}^*),~~~\forall x_i\in\mathbb{R}^{n_i}.             \label{P12}
\end{align}

To proceed, let $\nabla_i J_i(x_i,x_{-i})$ denote $\nabla_{x_i}J_i(x_i,x_{-i})$ for simplicity. It is then easy to see that an NE $x^*$ satisfies $\nabla_i J_i(x_i^*,x_{-i}^*)=0$ for all $i\in[N]$. Consequently, the Nash equilibrium seeking can be equivalently recast as to find fixed points of an operator $F$, defined by
\begin{align}
F&:=Id-r U,                                \label{P13}\\
U&:=col(\nabla_1 J_1,\ldots,\nabla_N J_N),          \label{P14}
\end{align}
where $r>0$ is any constant. By defining $\texttt{F}_i:=Id_i-r \nabla_i J_i$ for $i\in[N]$ with $Id_i:x\mapsto x_i$ for $x=col(x_1,\ldots,x_N)\in\mathbb{R}^n$, one can obtain that $F$ is block separable, i.e., $F=col(\texttt{F}_1,\ldots,\texttt{F}_N)$, which is consistent with problem (\ref{6b}). As a result, the linear convergence to an NE of the game can be assured by Theorem \ref{t3} under Assumptions \ref{a2}-\ref{a1} with $F$ being quasi-nonexpansive.

Note that Assumptions \ref{a2}-\ref{a1} are relativley mild, some of which are less conservative than those employed in the literature, as remarked below.

1) Assumption \ref{a0} is in fact equivalent to $\nabla_i J_i$ being Lipschitz for all $i\in[N]$, i.e., $\|\nabla_i J_i(x)-\nabla_i J_i(y)\|\leq q_i\|x-y\|$ for some $q_i>0$ and for all $x,y\in\mathbb{R}^n$, which has been frequently employed in the literature, see e.g., \cite{tatarenko2019geometric,bianchi2019distributed,bianchi2020fully,meng2020linear}. Then it can be readily obtained that $U$ is $q$-Lipschitz, where $q:=\sqrt{\sum_{i=1}^N q_i^2}$.

2) The linear regularity is strictly weaker than the strong monotonicity, which has been widely imposed for deriving the linear convergence \cite{tatarenko2019geometric,bianchi2019distributed,bianchi2020fully,meng2020linear}, i.e., $(U(x)-U(z))^\top(x-z)\geq \mu\|x-z\|^2$ for some $\mu>0$ and for all $x,z\in\mathbb{R}^n$. To see this, it is obvious that strong monotonicity is strictly stronger than quasi-strong monotonicity, i.e., $(U(x)-U(y))^\top(x-y)\geq \mu\|x-y\|^2$ for all $x\in\mathbb{R}^n$ and $y\in Fix(F)$. Meanwhile, quasi-strong monotonicity can imply the linear regularity of $F$, since it holds that $\|F(x)-x\|=r\|U(x)\|=r\|U(x)-U(P_{Fix(F)}(x))\|\geq r\mu\|x-P_{Fix(F)}(x)\|=r\mu d_{Fix(F)}(x)$ for all $x\in\mathbb{R}^n$, i.e., $d_{Fix(F)}(x)\leq \|F_\alpha(x)-x\|/(r\mu)$, where $U(P_{Fix(F)}(x))=0$ and the quasi-strong monotonicity have been utilized. It should be also noted that the game can have a closed convex set of NEs (not necessary to be unique) under linear regularity.

3) The quasi-nonexpansiveness of $F$ is a weak assumption. For example, if the aforementioned quasi-strong monotonicity holds, then it holds that for all $x\in\mathbb{R}^n$ and $y\in Fix(F)$,
\begin{align}
\|F(x)-y\|^2&=\|x-y-r(U(x)-U(y))\|^2          \nonumber\\
&=\|x-y\|^2-2r(x-y)^\top(U(x)-U(y))           \nonumber\\
&\hspace{0.4cm}+r^2\|U(x)-U(y)\|^2        \nonumber\\
&\leq (1-2\mu r+q^2r^2)\|x-y\|^2,            \label{P15}
\end{align}
where the quasi-strong monotonicity and $q$-Lipschitz continuity of $U$ are used in the inequality. In view of (\ref{P15}), it is easy to see that $F$ is even contractive, which is stronger than quasi-nonexpansive, if $r\in(0,\frac{2\mu}{q^2})$.

4) Assumption \ref{a2} requires the strong connectivity for directed graphs, which are not necessarily balanced. In contrast, balanced undirected/directed graphs are exploited in \cite{tatarenko2019geometric,bianchi2019distributed,bianchi2020fully,meng2020linear}. We note that time-varying graphs are considered in \cite{bianchi2020fully}, but the graphs still need to be balanced, and in this case, it is interesting for us to extend the results of this paper to time-varying communication graphs.

\section{Numerical Examples}\label{s5}

This section is to provide two numerical examples to corroborate the proposed algorithms.
\begin{example}\label{e1}
Consider a distributed optimization problem as discussed in Case 1 of Section \ref{app-1}, where $f_i(x)=\hbar_i(Ex)+b_i^\top x$, and $\hbar_i(z)$ is a strongly convex function with Lipschitz continuous gradient. It is easy to see that this problem is equivalent to finding a fixed point of the operator $F:=Id-\xi \nabla f$ for $\xi\in(0,2/L)$ (see Section \ref{app-1}), which is in the form (\ref{6}) with $F_i:=Id_i-\xi\nabla f_i$ for $i\in[N]$.

It should be noted that $f_i$ will be strongly convex when $E$ has full column rank, and $f_i$ will be convex but not strongly convex if $E$ does not have full column rank (Section III in \cite{banjac2018tight}), which is frequently encountered in practical applications, such as the $L1$-loss linear support vector machine (SVM) in machine learning \cite{wang2014iteration}. Denote by $X^*$ the nonempty set of optimizers of this problem. Although $f_i$ is not strongly convex when $E$ does not has full column rank, it has been shown in Theorem 18 of \cite{wang2014iteration} that this problem admits a global error bound, i.e., $d_{X^*}(x)\leq \tau\|\nabla f(x)\|,\forall x\in\mathbb{R}^n$ for some constant $\tau\geq0$, which further leads to $d_{Fix(F)}(x)\leq\frac{\tau}{\xi}\|x-F(x)\|$ for all $x\in\mathbb{R}^n$, i.e., satisfying the linear regularity condition.

\begin{figure}[H]
\centering
\includegraphics[width=2.8in]{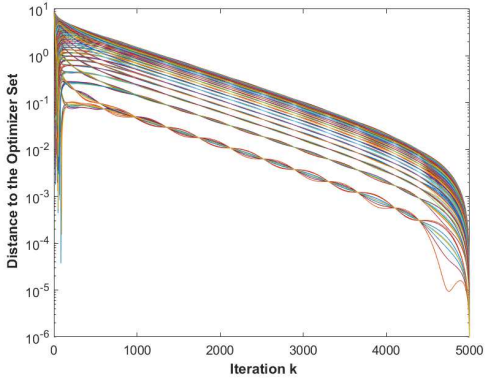}
\caption{Evolutions of distance to the optimizer set by DOT in this paper.}
\label{f1}
\end{figure}

\begin{figure}[H]
\centering
\includegraphics[width=2.8in]{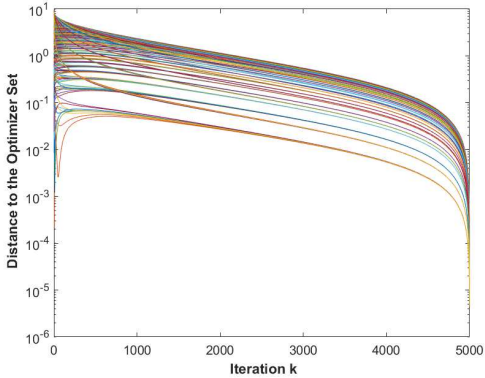}
\caption{Evolutions of distance to the optimizer set by D-KM in \cite{li2019distributed5}.}
\label{f2}
\end{figure}

In the simulation, let $N=100$, $n=5$, $\hbar_i(Ex)=|Ex-p_i|^2$, $E=(1,1,1,1,1)\in\mathbb{R}^{1\times 5}$, $p_i=i$, and $b_i=col(i,i,i,i,i)/5$ for all $i\in[N]$. Then it is easy to verify that the gradient constant of $f$ is $L=10$, thus holding $\xi\in (0,0.2)$. Setting $\alpha=0.05$ and $\xi=0.1$, running the DOT algorithm (\ref{9}) gives rise to the simulation results in Fig. \ref{f1}, indicating that all $x_{i,k}$'s converge linearly to the optimal set $X^*:=\{z=col(z_1,z_2)\in\mathbb{R}^2:z_1+2z_2=3(N+1)/8\approx 37.875\}$. In comparison with the D-KM iteration proposed in \cite{li2019distributed5} under the same communication graph as DOT (see Fig. \ref{f2}), which is equivalent to the classical distributed gradient descent (DGD) algorithm for this problem, it can be observed that the DOT algorithm here has a faster convergence speed. Overall, the simulation supports the theoretical result.
\end{example}

\begin{figure}[H]
\centering
\includegraphics[width=2.8in]{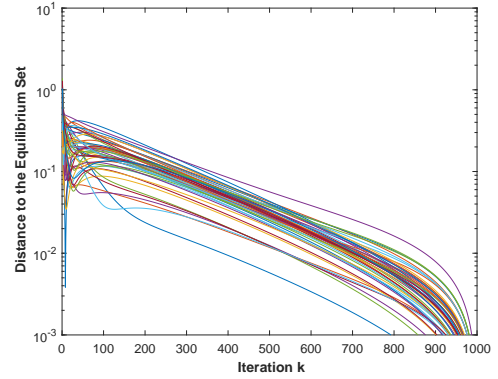}
\caption{Evolutions of $x_{j,k}^i-x_{j,k}^1$ for $i,j\in[N]$ by DOP.}
\label{f5}
\end{figure}

\begin{figure}[H]
\centering
\includegraphics[width=2.8in]{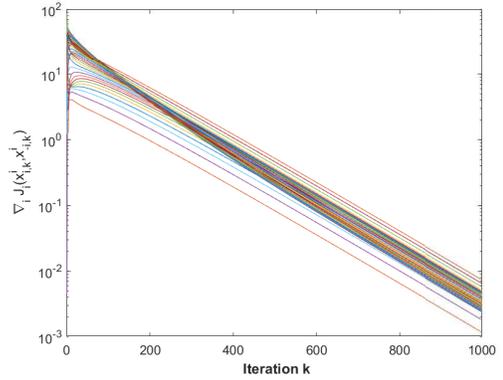}
\caption{Evolutions of $\nabla_i J_i(x_{i,k}^i,x_{-i,k}^i)$ for $i\in[N]$ by DOP.}
\label{f6}
\end{figure}

\begin{example}\label{e3}
Consider the class of games as discussed in Section \ref{app-3} with $N=50$ players. To be specific, each player $i$ has its decision vector in $\mathbb{R}^2$ with its cost function given as $J_i(x_i,x_{-i})=h_i(Ex_i)+l_i^\top x_i$, where $h_i(z)=r_i z^2+s_i z$ is strongly convex for $z\in\mathbb{R}$ with $r_i>0$ and $s_i\in\mathbb{R}$, and $l_i(x_{-i})=\sum_{j\neq i}c_{ij}x_j$ with $c_{ij}\in\mathbb{R}^{2\times 2}$. Note that $J_i$ is not strongly convex in $x_i$ if $E$ is not of full column rank, as discussed in Example \ref{e1}. In this example, let $E=(1,1)$, which does not have full column rank. Thus, $J_i$ is not strongly convex in $x_i$, however $\texttt{F}_i:=Id_i-r \nabla_i J_i$ is linearly regular as similarly illustrated as in Example \ref{e1}. It is then easy to verify that the global operator $F=(\texttt{F}_1,\ldots,\texttt{F}_N)$ is linearly regular. Moreover, the Nash equilibrium may not be unique since $E$ is not of full column rank. Set $\alpha=0.01$ and $r=0.1$. By randomly choosing $r_i$, $s_i$, and $c_{ij}$ for $i,j\in[N]$ with a randomly generated strongly connected communication graph, performing the developed DOP with each component of initial conditions randomly in $[0,1]$ gives the simulation results in Figs. \ref{f5} and \ref{f6}. In Fig. \ref{f5}, the distances from $x_k^i=col(x_{1,k}^i,\ldots,x_{N,k}^i)$ to the equilibrium set are plotted for all players, showing that all players' estimate $x_k^i$'s converge to the equilibrium set at a linear rate. On the other hand, the gradient $\nabla_iJ_i(x_{i,k}^i,x_{-i,k}^i)$ of each agent $i$ is given in Fig. \ref{f6}, indicating that all agents' gradients converge linearly to zero. Hence, the simulation results support the theoretical result in Theorem \ref{t3}.
\end{example}

\section{Conclusion}\label{s6}

This paper has investigated the fixed point seeking problem for a quasi-nonexpansive global operator over a fixed, unbalanced, and directed communication graph, for which two scenarios have been considered, i.e., the global operator is respectively sum separable and block separable under the linear regularity condition. For the first case, the global operator consists of a sum of local operators, which are assumed to be Lipschitz. To solve this case, a distributed algorithm, DOT, has been proposed and shown to be convergent to a fixed point of the global operator at a linear rate. For the second case, a distributed algorithm, DOP, has been developed, showing to be linearly convergent to a fixed point of the global operator. Meanwhile, two applications have been presented in detail, i.e., distributed optimization and multi-player game under partial-decision information. In future, it is interesting to study asynchronous algorithms, non-identical stepsizes for agents, and time-varying communication graphs.

\section*{Acknowledgment}

The authors are grateful to the Editor, the Associate Editor and the anonymous reviewers for their insightful suggestions.

\section*{Appendix}

%\subsection{Proof of Proposition \ref{prop1}}
%
%The necessity part is the same to that of Theorem 1 in \cite{banjac2018tight}, thus omitted here. For the sufficiency part, it can be obtained that for any $x\in \mathbb{S}$ and $y\in Fix(F)$,
%\begin{align}
%&\|F_\alpha(x)-y\|^2=\|(1-\alpha)(x-y)+\alpha(F(x)-y)\|^2         \nonumber\\
%&\hspace{0.3cm}=(1-\alpha)\|x-y\|^2+\alpha\|F(x)-y\|^2                      \nonumber\\
%&\hspace{0.7cm}-\alpha(1-\alpha)\|F(x)-x\|^2       \nonumber\\
%&\hspace{0.3cm}\leq [1+\alpha(\delta^2-1)]\|x-y\|^2-\frac{1-\alpha}{\alpha}\|F_\alpha(x)-x\|^2,        \nonumber
%\end{align}
%where the second equality has used the fact that $\|rz_1+(1-r)z_2\|^2=r\|z_1\|^2+(1-r)\|z_2\|^2-r(1-r)\|z_1-z_2\|^2$ for all $z_1,z_2\in\mathcal{H}$ and any $r\in\mathbb{R}$ (e.g., Lemma 7 in \cite{li2019distributed3}), and the inequality has leveraged Assumption \ref{a-c} and $F_\alpha(x)-x=\alpha(F(x)-x)$. Invoking the above inequality and similar arguments to the sufficiency proof of Theorem 1 in \cite{banjac2018tight} can yield that
%\begin{align}
%d_{Fix(F)}(F_\alpha(x))\leq\sqrt{1+\alpha(\delta^2-1)-\frac{1-\alpha}{\alpha\kappa^2}}d_{Fix(F)}(x),       \nonumber
%\end{align}
%where $\alpha(\delta^2-1)<(1-\alpha)/(\alpha\kappa^2)$ when $\alpha\in(0,\min\{1,\alpha_B\})$. This finishes the proof.

\subsection{Proof of Lemma \ref{l3}}

To prove (\ref{cdc9-1}), it can be obtained that
\begin{align}
\|\mathbf{A}z-\mathbf{A}_\infty z\|_\pi&=\|(\mathbf{A}-\mathbf{A}_\infty)(z-\mathbf{A}_\infty z)\|_\pi         \nonumber\\
&\leq \|\mathbf{A}-\mathbf{A}_\infty\|_\pi\|z-\mathbf{A}_\infty z\|_\pi,                 \label{cdc10}
\end{align}
where the equality has used the fact that $\mathbf{A}\mathbf{A}_\infty=\mathbf{A}_\infty\mathbf{A}_\infty=\mathbf{A}_\infty$.

Consider the term $\|\mathbf{A}-\mathbf{A}_\infty\|_\pi$ in (\ref{cdc10}). To do so, by definition (\ref{cdc3}), one has that for any $x\in\mathbb{R}^N$ and $y\in\mathcal{H}$,
\begin{align}
\|x\otimes y\|_\pi&=\sqrt{\sum_{i=1}^N \pi_i x_i^2\|y\|^2}=\|y\|\sqrt{\sum_{i=1}^N \pi_i x_i^2}=\|x\|_\pi \|y\|,                                           \label{cdc11}
\end{align}
which, together with the norm's definition, leads to
\begin{align}
\|\mathbf{A}-\mathbf{A}_\infty\|_\pi&=\sup_{\|x\otimes y\|_\pi\neq 0}\frac{\|(\mathbf{A}-\mathbf{A}_\infty)(x\otimes y)\|_\pi}{\|x\otimes y\|_\pi}         \nonumber\\
&=\sup_{\|x\otimes y\|_\pi\neq 0}\frac{\|[(A-A_\infty)x]\otimes y\|_\pi}{\|x\otimes y\|_\pi}         \nonumber\\
&=\sup_{\|x\|_\pi\|y\|\neq 0}\frac{\|(A-A_\infty)x\|_\pi \|y\|}{\|x\|_\pi \|y\|}         \nonumber\\
&=\sup_{\|x\|_\pi\neq 0}\frac{\|(A-A_\infty)x\|_\pi}{\|x\|_\pi}         \nonumber\\
&=\|A-A_\infty\|_\pi.                                                \label{cdc12}
\end{align}
Note that $\|A-A_\infty\|_\pi=\rho_1$ by Lemma \ref{l2}. Consequently, putting together (\ref{cdc10})-(\ref{cdc12}) gives rise to (\ref{cdc9-1}). By noting that $\mathbf{B}\mathbf{B}_\infty=\mathbf{B}_\infty\mathbf{B}_\infty=\mathbf{B}_\infty$, similar arguments can be applied to obtain (\ref{cdc9-2}) and (\ref{cdc9-3}). This ends the proof.

\subsection{Proof of Theorem \ref{t1}}

Let us first establish upper bounds on $\|x_{k+1}-\mathbf{A}_\infty x_{k+1}\|_\pi$, $\|x_{k+1}-x_k\|$, $\|y_{k+1}-\mathbf{B}_\infty y_{k+1}\|_\nu$, and $\|\mathbf{A}_\infty x_{k+1}-{\bf 1}_N\otimes x_{k+1}^*)\|$, where $\bar{x}_k:=\sum_{i=1}^N \pi_i x_{i,k}$ and $x_k^*:=P_{Fix(F)}(\bar{x}_k)$ for all $k\geq 0$.

For $\|x_{k+1}-\mathbf{A}_\infty x_{k+1}\|_\pi$, by noting $\mathbf{A}_\infty \mathbf{A}=\mathbf{A}_\infty$, invoking (\ref{cf1}) yields that
\begin{align}
&\|x_{k+1}-\mathbf{A}_\infty x_{k+1}\|_\pi          \nonumber\\
&=\|(1-\alpha)\mathbf{A} x_k+\frac{\alpha}{N}(D_\nu^{-1}\otimes Id)y_k-(1-\alpha)\mathbf{A}_\infty \mathbf{A}x_k         \nonumber\\
&\hspace{0.4cm}-\frac{\alpha}{N}\mathbf{A}_\infty(D_\nu^{-1}\otimes Id)y_k\|_\pi                                    \nonumber\\
&\leq (1-\alpha)\|\mathbf{A}x_k-\mathbf{A}_\infty x_k\|_\pi                                                 \nonumber\\
&\hspace{0.4cm}+\frac{\alpha}{N}\|(I_N\otimes Id-\mathbf{A}_\infty)(D_\nu^{-1}\otimes Id)y_k\|_\pi.         \label{pf1}
\end{align}
Consider the last term in (\ref{pf1}). One can obtain that
\begin{align}
&\|(I_N\otimes Id-\mathbf{A}_\infty)(D_\nu^{-1}\otimes Id)y_k\|_\pi           \nonumber\\
&=\|(I_N\otimes Id-\mathbf{A}_\infty)[(D_\nu^{-1}\otimes Id)y_k-{\bf 1}_N\otimes \bar{y}_k]\|_\pi       \nonumber\\
&=\|(I_N\otimes Id-\mathbf{A}_\infty)(D_\nu^{-1}\otimes Id)[y_k-\nu\otimes \bar{y}_k]\|_\pi       \nonumber\\
&\leq \|I_N\otimes Id-\mathbf{A}_\infty\|_\pi\|D_\nu^{-1}\otimes Id\|_\pi\|y_k-\mathbf{B}_\infty y_k\|_\pi       \nonumber\\
&\leq c_4\|D_\nu^{-1}\|_\pi\|y_k-\mathbf{B}_\infty y_k\|_\nu,                                                 \label{pf2}
\end{align}
where $\bar{y}_k$ is defined in Lemma \ref{l6}, the first inequality has applied the fact that $\nu\otimes \bar{y}_k=\mathbf{B}_\infty y_k$, and the last inequality has employed (\ref{cdc9-3}), $\|D_\nu^{-1}\otimes Id\|_\pi=\|D_\nu^{-1}\|_\pi$ (using the same argument as that in Lemma \ref{l3}), and (\ref{eq2}).

In view of (\ref{eq1}) and (\ref{cdc9-1}), inserting (\ref{pf2}) in (\ref{pf1}) results in
\begin{align}
\|x_{k+1}-\mathbf{A}_\infty x_{k+1}\|_\pi&\leq(1-\alpha)\rho_1\|x_k-\mathbf{A}_\infty x_k\|_\pi         \nonumber\\
&\hspace{0.1cm}+\frac{\alpha c_2c_4\|D_\nu^{-1}\|}{N}\|y_k-\mathbf{B}_\infty y_k\|_\nu.                \label{pf3}
\end{align}

As for $\|x_{k+1}-x_k\|$, it can be obtained from (\ref{cf1}) that
\begin{align}
&\|x_{k+1}-x_k\|          \nonumber\\
&=\|\mathbf{A}x_k-x_k+\alpha\Big(\frac{D_\nu^{-1}\otimes Id}{N}y_k-\mathbf{A}x_k\Big)\|        \nonumber\\
&\leq \|A-I\|\|x_k-\mathbf{A}_\infty x_k\|+\frac{\alpha\|D_\nu^{-1}\|}{N}\|y_k-\mathbf{B}_\infty y_k\|        \nonumber\\
&\hspace{0.4cm}+\alpha\|\frac{D_\nu^{-1}\otimes Id}{N}\mathbf{B}_\infty y_k-\mathbf{A}x_k\|,                 \label{pf4}
\end{align}
where the inequality has leveraged the triangle inequality and the facts that $(\mathbf{A}-I_N\otimes Id)(x_k-\mathbf{A}_\infty x_k)=\mathbf{A}x_k-x_k$ and $\|\mathbf{A}-I_N\otimes Id\|=\|A-I\|$ (using the same argument as that in Lemma \ref{l3}).

Consider the last term in (\ref{pf4}). By using $\mathbf{B}_\infty=\nu{\bf 1}_N^\top\otimes Id$, one has that
\begin{align}
&\|\frac{D_\nu^{-1}\otimes Id}{N}\mathbf{B}_\infty y_k-\mathbf{A}x_k\|=\|{\bf 1}_N\otimes\frac{\bar{y}_k}{N}-\mathbf{A}x_k\|           \nonumber\\
&\leq \|{\bf 1}_N\otimes \Big(\frac{\bar{y}_k}{N}-\frac{\sum_{i=1}^N F_i(\bar{x}_k)}{N}\Big)\|             \nonumber\\
&\hspace{0.4cm}+\|{\bf 1}_N\otimes \Big(\frac{\sum_{i=1}^N F_i(\bar{x}_k)}{N}-x_k^*\Big)\|         \nonumber\\
&\hspace{0.4cm}+\|{\bf 1}_N\otimes x_k^*-\mathbf{A}_\infty x_k\|+\|\mathbf{A}_\infty x_k-\mathbf{A}x_k\|.         \label{pf5}
\end{align}

For the first term in the last inequality in (\ref{pf5}), invoking Lemma \ref{l6}, one can obtain that
\begin{align}
&\|{\bf 1}_N\otimes \Big(\frac{\bar{y}_k}{N}-\frac{\sum_{i=1}^N F_i(\bar{x}_k)}{N}\Big)\|^2         \nonumber\\
&=\|{\bf 1}_N\otimes \frac{\sum_{i=1}^N (F_i(x_{i,k})-F_i(\bar{x}_k))}{N}\|^2                  \nonumber\\
&=\frac{1}{N}\|\sum_{i=1}^N (F_i(x_{i,k})-F_i(\bar{x}_k))\|^2                             \nonumber\\
&\leq \sum_{i=1}^N \|F_i(x_{i,k})-F_i(\bar{x}_k)\|^2                            \nonumber\\
&\leq \sum_{i=1}^N L_i^2\|x_{i,k}-\bar{x}_k\|^2                                    \nonumber\\
&\leq \bar{L}^2\|x_k-\mathbf{A}_\infty x_k\|^2,                                       \label{pf6}
\end{align}
where the first and second inequalities have employed $\|\sum_{i=1}^N z_i\|^2\leq N\sum_{i=1}^N \|z_i\|^2$ for any vectors $z_i$'s and Assumption \ref{a0}, respectively. Similarly, it can be obtained that
\begin{align}
\|{\bf 1}_N\otimes \big(\frac{\sum_{i=1}^N F_i(\bar{x}_k)}{N}-x_k^*\big)\|^2&=N\|F(\bar{x}_k)-x_k^*\|^2     \nonumber\\
&\hspace{-0.8cm}\leq N \|\bar{x}_k-x_k^*\|^2            \nonumber\\
&\hspace{-0.8cm}=\|\mathbf{A}_\infty x_k-{\bf 1}_N\otimes x_k^*\|^2,          \label{cdc0}
\end{align}
where $x_k^*\in Fix(F)$ and the quasi-nonexpansiveness of $F$ have been used in the inequality.

As a result, substituting (\ref{pf6}) and (\ref{cdc0}) into (\ref{pf5}) leads to
\begin{align}
&\|\frac{D_\nu^{-1}\otimes Id}{N}\mathbf{B}_\infty y_k-\mathbf{A}x_k\|          \nonumber\\
&\leq \bar{L}\|x_k-\mathbf{A}_\infty x_k\|+2\|\mathbf{A}_\infty x_k-{\bf 1}_N\otimes x_k^*\|       \nonumber\\
&\hspace{0.4cm}+\|\mathbf{A}x_k-\mathbf{A}_\infty x_k\|              \nonumber\\
&\leq \frac{\rho_1+\bar{L}}{c_1}\|x_k-\mathbf{A}_\infty x_k\|_\pi+2\|\mathbf{A}_\infty x_k-{\bf 1}_N\otimes x_k^*\|,              \label{pf7}
\end{align}
where (\ref{eq1}) and (\ref{cdc5}) have been utilized in the last inequality. Putting (\ref{pf7}) in (\ref{pf4}) leads to
\begin{align}
\|x_{k+1}-x_k\|&\leq \frac{\alpha(\rho_1+\bar{L})+\|A-I\|}{c_1}\|x_k-\mathbf{A}_\infty x_k\|_\pi           \nonumber\\
&\hspace{0.4cm}+\frac{\alpha c_4\|D_\nu^{-1}\|}{Nc_1}\|y_k-\mathbf{B}_\infty y_k\|_\nu        \nonumber\\
&\hspace{0.4cm}+2\alpha\|\mathbf{A}_\infty x_k-{\bf 1}_N\otimes x_k^*\|.                 \label{pf8}
\end{align}

Regarding $\|y_{k+1}-\mathbf{B}_\infty y_{k+1}\|_\nu$, invoking (\ref{cf2}) yields that
\begin{align}
\|y_{k+1}-\mathbf{B}_\infty y_{k+1}\|_\nu&=\|\mathbf{B}y_k-\mathbf{B}_\infty\mathbf{B}y_k+\mathbf{F}(x_{k+1})               \nonumber\\
&\hspace{-0.1cm}-\mathbf{F}(x_k)-\mathbf{B}_\infty(\mathbf{F}(x_{k+1})-\mathbf{F}(x_k))\|_\nu          \nonumber\\
&\hspace{-1.0cm}\leq \|\mathbf{B}y_k-\mathbf{B}_\infty y_k\|_\nu+\|\mathbf{F}(x_{k+1})-\mathbf{F}(x_k)\|_\nu                         \nonumber\\
&\hspace{-0.6cm}+\|\mathbf{B}_\infty(\mathbf{F}(x_{k+1})-\mathbf{F}(x_k))\|_\nu                        \nonumber\\
&\hspace{-1.0cm}\leq \rho_2\|y_k-\mathbf{B}_\infty y_k\|_\nu                                \nonumber\\
&\hspace{-0.6cm}+\frac{c_2(1+\sqrt{N})}{c_3}\|\mathbf{F}(x_{k+1})-\mathbf{F}(x_k)\|,                           \label{pf9}
\end{align}
where $\mathbf{B}_\infty\mathbf{B}=\mathbf{B}_\infty$ has been used in the first inequality, and (\ref{eq1}), (\ref{eq2}), (\ref{cdc9-2}) and $\|\mathbf{B}_\infty\|=\|B_\infty\|\leq \sqrt{\|B_\infty\|_1\|B_\infty\|_\infty}\leq \sqrt{N}$ have been exploited in the last inequality.

On the other hand, it can be obtained that
\begin{align}
\|\mathbf{F}(x_{k+1})-\mathbf{F}(x_k)\|^2&=\sum_{i=1}^N\|F_i(x_{i,k+1})-F_i(x_{i,k})\|^2          \nonumber\\
&\leq \sum_{i=1}^N L_i^2\|x_{i,k+1}-x_{i,k}\|^2                 \nonumber\\
&\leq \bar{L}^2\|x_{k+1}-x_k\|^2,                                   \label{pf10}
\end{align}
where Assumption \ref{a0} has been applied to obtain the first inequality in (\ref{pf10}).

Now substituting (\ref{pf8}) and (\ref{pf10}) into (\ref{pf9}) leads to
\begin{align}
&\|y_{k+1}-\mathbf{B}_\infty y_{k+1}\|_\nu              \nonumber\\
&\leq \rho_2\|y_k-\mathbf{B}_\infty y_k\|_\nu+\frac{c_2\bar{L}(\sqrt{N}+1)}{c_3}\|x_{k+1}-x_k\|                         \nonumber\\
&\leq \Big(\rho_2+\frac{\alpha c_2c_4\bar{L}(\sqrt{N}+1)\|D_\nu^{-1}\|}{Nc_1c_3}\Big)\|y_k-\mathbf{B}_\infty y_k\|_\nu                         \nonumber\\
&\hspace{0.4cm}+\frac{c_2\bar{L}(\sqrt{N}+1)}{c_1c_3}(\alpha\theta_3+\|A-I\|)\|x_k-\mathbf{A}_\infty x_k\|_\pi                          \nonumber\\
&\hspace{0.4cm}+\frac{2\alpha c_2\bar{L}(\sqrt{N}+1)}{c_3}\|\mathbf{A}_\infty x_k-{\bf 1}_N\otimes x_k^*\|,                      \label{pf12}
\end{align}
where $\theta_3=\rho_1+\bar{L}$.

With regard to $\|\mathbf{A}_\infty x_{k+1}-{\bf 1}_N\otimes P_{Fix(F)}(\bar{x}_{k+1})\|$, by defining $\bar{x}_k^*:=P_{Fix(F)}(F_\alpha(\bar{x}_k))$ and noting that $\mathbf{A}_\infty x_{k+1}={\bf 1}_N\otimes \bar{x}_{k+1}$, $x_{k+1}^*=P_{Fix(F)}(\bar{x}_{k+1})$ and $\mathbf{A}_\infty\mathbf{A}=\mathbf{A}_\infty$, invoking (\ref{cf1}) gives rise to
\begin{align}
&\|\mathbf{A}_\infty x_{k+1}-{\bf 1}_N\otimes x_{k+1}^*\|         \nonumber\\
&\leq \|\mathbf{A}_\infty x_{k+1}-{\bf 1}_N\otimes \bar{x}_k^*\|           \nonumber\\
&=\|\mathbf{A}_\infty\mathbf{A}x_k+\alpha\big[\frac{\mathbf{A}_\infty}{N}(D_\nu^{-1}\otimes Id)y_k-\mathbf{A}_\infty\mathbf{A}x_k\big]      \nonumber\\
&\hspace{0.4cm}-{\bf 1}_N\otimes \bar{x}_k^*\|         \nonumber\\
&=\|{\bf 1}_N\otimes \bar{x}_k+\alpha\big[\frac{\mathbf{A}_\infty}{N}({\bf 1}_N\otimes \bar{y}_k)-{\bf 1}_N\otimes \bar{x}_k\big]-{\bf 1}_N\otimes \bar{x}_k^*            \nonumber\\
&\hspace{0.4cm}+\frac{\alpha\mathbf{A}_\infty}{N}(D_\nu^{-1}\otimes Id)(y_k-\mathbf{B}_\infty y_k)\|            \nonumber\\
&\leq \|{\bf 1}_N\otimes \bar{x}_k+\alpha\big[\frac{\mathbf{A}_\infty}{N}({\bf 1}_N\otimes \bar{y}_k)-{\bf 1}_N\otimes \bar{x}_k\big]-{\bf 1}_N\otimes \bar{x}_k^*\|            \nonumber\\
&\hspace{0.4cm}+\frac{\alpha c_4 \|D_\nu^{-1}\|}{\sqrt{N}c_1}\|y_k-\mathbf{B}_\infty y_k\|_\nu,            \label{pf13}
\end{align}
where the last inequality has utilized (\ref{eq1}), (\ref{eq2}), and the fact that $\|\mathbf{A}_\infty\|=\|A_\infty\|\leq \sqrt{\|A_\infty\|_1\|A_\infty\|_\infty}\leq \sqrt{N}$.

On the other hand, by $\mathbf{A}_\infty={\bf 1}_N\otimes \pi^\top$ and Lemma \ref{l6}, one can obtain that
\begin{align}
&\|{\bf 1}_N\otimes \bar{x}_k+\alpha\big[\frac{\mathbf{A}_\infty}{N}({\bf 1}_N\otimes \bar{y}_k)-{\bf 1}_N\otimes \bar{x}_k\big]-{\bf 1}_N\otimes \bar{x}_k^*\|      \nonumber\\
&=\|{\bf 1}_N\otimes \bar{x}_k+\alpha\big[{\bf 1}_N\otimes \frac{\sum_{i=1}^N F_i(x_{i,k})}{N}-{\bf 1}_N\otimes \bar{x}_k\big]       \nonumber\\
&\hspace{0.4cm}-{\bf 1}_N\otimes \bar{x}_k^*\|        \nonumber\\
&=\|{\bf 1}_N\otimes F_\alpha(\bar{x}_k)-{\bf 1}_N\otimes \bar{x}_k^*               \nonumber\\
&\hspace{0.4cm}+\alpha{\bf 1}_N\otimes \frac{\sum_{i=1}^N (F_i(x_{i,k})-F_i(\bar{x}_k))}{N}\|                \nonumber\\
&\leq \|{\bf 1}_N\otimes [F_\alpha(\bar{x}_k)-\bar{x}_k^*]\|              \nonumber\\
&\hspace{0.4cm}+\alpha\|{\bf 1}_N\otimes \frac{\sum_{i=1}^N (F_i(x_{i,k})-F_i(\bar{x}_k))}{N}\|,                   \label{pf14}
\end{align}
where $F_\alpha$ is defined in (\ref{FA}).

For the term $\|{\bf 1}_N\otimes [F_\alpha(\bar{x}_k)-\bar{x}_k^*]\|$ in (\ref{pf14}), invoking $Fix(F_\alpha)=Fix(F)$ and Lemma \ref{l1} implies that
\begin{align}
\|{\bf 1}_N\otimes [F_\alpha(\bar{x}_k)-\hat{x}_k^*]\|^2&=N\|F_\alpha(\bar{x}_k)-\bar{x}_k^*\|^2            \nonumber\\
&\leq N\rho_3^2\|\bar{x}_k-P_{Fix(F)}(\bar{x}_k)\|^2           \nonumber\\
&= N\rho_3^2\|\bar{x}_k-x_k^*\|^2           \nonumber\\
&=\rho_3^2\|{\bf 1}_N\otimes (\bar{x}_k-x_k^*)\|^2.                \label{pf15}
\end{align}

Now, putting together (\ref{pf6}) and (\ref{pf13})-(\ref{pf15}) results in
\begin{align}
&\|\mathbf{A}_\infty x_{k+1}-{\bf 1}_N\otimes x_{k+1}^*\|         \nonumber\\
&\leq \rho_3\|\mathbf{A}_\infty x_k-{\bf 1}_N\otimes x_k^*\|+\frac{\alpha \bar{L}}{c_1}\|x_k-\mathbf{A}_\infty x_k\|_\pi       \nonumber\\
&\hspace{0.4cm}+\frac{\alpha c_4\|D_\nu^{-1}\|}{\sqrt{N}c_1}\|y_k-\mathbf{B}_\infty y_k\|_\nu.            \label{pf16}
\end{align}

In summary, let us define $z_k:=col(\|x_k-\mathbf{A}_\infty x_k\|_\pi,\|y_k-\mathbf{B}_\infty y_k\|_\nu,\|\mathbf{A}_\infty x_k-{\bf 1}_N\otimes x_k^*\|)$. Combining (\ref{pf3}), (\ref{pf12}), and (\ref{pf16}) with $\alpha\in(0,1)$ yields that
\begin{align}
z_{k+1}\leq M(\alpha)z_k,               \label{pf17}
\end{align}
where $M(\alpha)$ is defined in (\ref{cdc14}).

It is easy to see that $z_k$ will converge to the origin at an exponential rate if $\rho(M(\alpha))<1$. To ensure $\rho(M(\alpha))<1$, it is straightforward to observe that when $\alpha=0$, $1$ is a simple eigenvalue of $M(0)$ with corresponding left and right eigenvectors both being $col(0,0,1)$. Then invoking Lemma \ref{l-m} gives rise to
\begin{align}
\frac{d\lambda(\alpha)}{d\alpha}\Big|_{\alpha=0}=-\frac{1}{4\kappa^2}<0,          \nonumber
\end{align}
indicating that the simple eigenvalue $1$ of $M(0)$ will decrease when increasing the value of $\alpha$. Thus, by the continuity of $\rho(M(\alpha))$ with respect to $\alpha$, there must exist a constant $\alpha_c>0$ such that $\rho(M(\alpha))<1$ for all $\alpha\in(0,\alpha_c)$. To find $\alpha_c$, one can see that the graph associated with $M(\alpha)$ consisting of $3$ agents is strongly connected, which, in conjunction with Theorem C.3 in \cite{ren2008distributed}, leads to $M(\alpha)$ being irreducible. Furthermore, in view of Lemma \ref{l4}, it can be obtained that $M(\alpha)$ is primitive, which together with Lemma \ref{l5} results in that $\rho(M(\alpha))$ is a simple eigenvalue of $M(\alpha)$ and all other eigenvalues have absolute values of less than $\rho(M(\alpha))$. Moreover, it can be ensured that $\rho(M(\alpha))=1$ when increasing $\alpha$ from $0$ to some point, and thereby the value of $\alpha_c$ can be calculated by letting $det(I-M(\alpha))=0$. This completes the proof.

\subsection{Proof of Theorem \ref{t3}}

Let us bound $\|x_{k+1}-A_\infty x_{k+1}\|_\pi$ and $d_{Fix(F)}(\tilde{x}_{k+1})$ in the following.

First, to bound $\|x_{k+1}-A_\infty x_{k+1}\|_\pi$, in view of (\ref{P1}) and $A_\infty A=A_\infty$, one has that
\begin{align}
\|x_{k+1}-A_\infty x_{k+1}\|_\pi&=\|Ax_k+\alpha\bar{F}-A_\infty A x_k-\alpha A_\infty\bar{F}\|_\pi          \nonumber\\
&\hspace{-1.1cm}\leq \|Ax_k-A_\infty x_k\|_\pi+\alpha\|\bar{F}-{\bf 1}_N\tilde{F}\|_\pi          \nonumber\\
&\hspace{-1.1cm}\leq \rho_1\|x_k-A_\infty x_k\|_\pi+\alpha c_2\|\bar{F}-{\bf 1}_N\tilde{F}\|,        \label{qq6}
\end{align}
where (\ref{eq1}) and (\ref{P3}) have been utilized in the last inequality.

For the last term in (\ref{qq6}), it is easy to verify that
{\small\begin{align}
\bar{F}-{\bf 1}_N\tilde{F}=\left(
                             \begin{array}{cccc}
                               (\frac{1}{\pi_1}-1)e_{1,k} & -e_{2,k} & \cdots & -e_{N,k} \\
                               -e_{1,k} & (\frac{1}{\pi_2}-1)e_{2,k} & \cdots & -e_{N,k} \\
                               \vdots & \vdots & \ddots & \vdots \\
                               -e_{1,k} & -e_{2,k} & \cdots & (\frac{1}{\pi_N}-1)e_{N,k} \\
                             \end{array}
                           \right)                  \nonumber
\end{align}}
with $e_{i,k}:=\texttt{F}_i(x_k^i)-\hat{x}_{i,k}$, and thus it can be obtained that
\begin{align}
&\|\bar{F}-{\bf 1}_N\tilde{F}\|^2=\sum_{i=1}^N\big[(\frac{1}{\pi_i}-1)^2\|e_{i,k}\|^2+\sum_{j\neq i}\|e_{j,k}\|^2\big]     \nonumber\\
&\hspace{0.4cm}\leq \varpi\sum_{i=1}^N\|e_{i,k}\|^2         \nonumber\\
&\hspace{0.4cm}=\varpi\sum_{i=1}^N\|\texttt{F}_i(x_k^i)-\texttt{F}_i(\tilde{x}_k)+\texttt{F}_i(\tilde{x}_k)-x_i^*+x_i^*-\tilde{x}_{i,k}        \nonumber\\
&\hspace{0.8cm}+\tilde{x}_{i,k}-\hat{x}_{i,k}\|^2          \nonumber\\
&\hspace{0.4cm}\leq 4\varpi\sum_{i=1}^N\big(\|\texttt{F}_i(x_k^i)-\texttt{F}_i(\tilde{x}_k)\|^2+\|\texttt{F}_i(\tilde{x}_k)-x_i^*\|^2        \nonumber\\
&\hspace{0.8cm}+\|\tilde{x}_{i,k}-x_i^*\|^2+\|\hat{x}_{i,k}-\tilde{x}_{i,k}\|^2\big),          \label{qq7}
\end{align}
where $\varpi:=N-1+(1-\underline{\pi})^2/\underline{\pi}^2$, $x^*=(x_1^*,\ldots,x_N^*)$ denotes any fixed point of $F$, $\pi_i\geq\underline{\pi}$ has been used in the first inequality and $\|\sum_{i=1}^n z_i\|^2\leq n\sum_{i=1}^n \|z_i\|^2$ for any vectors $z_i$'s in the last inequality. Note that $\tilde{x}_{i,k}$ is the $i$-th block-coordinate of $\tilde{x}_k$ defined in (\ref{P2}).

To proceed, one can obtain that
\begin{align}
&\sum_{i=1}^N\|\hat{x}_{i,k}-\tilde{x}_{i,k}\|^2=\sum_{i=1}^N\|\sum_{j=1}^N a_{ij}(x_{i,k}^j-\tilde{x}_{i,k})\|^2     \nonumber\\
&\leq \sum_{i=1}^N\sum_{j=1}^N a_{ij}\|x_{i,k}^j-\tilde{x}_{i,k}\|^2             \nonumber\\
&\leq \sum_{i=1}^N\sum_{j=1}^N \|x_{i,k}^j-\tilde{x}_{i,k}\|^2             \nonumber\\
&=\|x_k-{\bf 1}_N\tilde{x}_k\|^2,                                     \label{qq8}
\end{align}
where $\sum_{j=1}^N a_{ij}=1$ has been exploited for the first equality, the convexity of norm $\|\cdot\|^2$ for the first inequality, and $a_{ij}\leq 1$ for the second inequality.

Now, invoking Assumptions \ref{a0} and \ref{a1}.1, (\ref{qq7}) and (\ref{qq8}) yields
\begin{align}
\|\bar{F}-{\bf 1}_N\tilde{F}\|^2&\leq 4\varpi(\bar{L}^2+1)\|x_k-{\bf 1}_N\tilde{x}_k\|^2        \nonumber\\
&\hspace{0.4cm}+8\varpi\|\tilde{x}_{k}-x^*\|^2,          \label{qq9}
\end{align}
which, together with ${\bf 1}_N\tilde{x}_k=A_\infty x_k$, implies
\begin{align}
\|\bar{F}-{\bf 1}_N\tilde{F}\|&\leq 2\sqrt{\varpi(\bar{L}^2+1)}\|x_k-A_\infty x_k\|        \nonumber\\
&\hspace{0.4cm}+2\sqrt{2\varpi}\|\tilde{x}_{k}-x^*\|.          \label{qq10}
\end{align}

At this step, by choosing $x^*=P_{Fix(F)}(\tilde{x}_k)$, substituting (\ref{qq10}) into (\ref{qq6}) leads to
\begin{align}
\|x_{k+1}-A_\infty x_{k+1}\|_\pi&\leq (\rho_1+\alpha\theta_3)\|x_k-A_\infty x_k\|_\pi           \nonumber\\
&\hspace{0.0cm}+2\alpha c_2\sqrt{2\varpi}d_{Fix(F)}(\tilde{x}_k),        \label{qq11}
\end{align}
where $\theta_3:=2c_2\sqrt{N(\bar{L}^2+1)}/c_1$.

Second, to bound $d_{Fix(F)}(\tilde{x}_{k+1})$, one can first observe that
\begin{align}
\tilde{F}&=(e_{1,k},\ldots,e_{N,k})       \nonumber\\
&=(\texttt{F}_1(\tilde{x}_k)-\tilde{x}_{1,k},\ldots,\texttt{F}_N(\tilde{x}_k)-\tilde{x}_{N,k})          \nonumber\\
&\hspace{0.4cm}+h_{1,k}+h_{2,k},                    \label{qq12}
\end{align}
where $e_{i,k}:=\texttt{F}_i(x_k^i)-\hat{x}_{i,k}$ and
\begin{align}
h_{1,k}&:=(\texttt{F}_1(x_k^1)-\texttt{F}_1(\tilde{x}_k),\ldots,\texttt{F}_N(x_k^N)-\texttt{F}_N(\tilde{x}_k)),        \nonumber\\
h_{2,k}&:=(\tilde{x}_{1,k}-\hat{x}_{1,k},\ldots,\tilde{x}_{N,k}-\hat{x}_{N,k}).         \nonumber
\end{align}
Meanwhile, invoking Assumption \ref{a0} yields that
\begin{align}
\|h_{1,k}\|^2&=\sum_{i=1}^N \|\texttt{F}_i(x_k^i)-\texttt{F}_i(\tilde{x}_k)\|^2\leq \sum_{i=1}^N L_i^2\|x_k^i-\tilde{x}_k\|^2                  \nonumber\\
&\leq \bar{L}^2\|x_k-{\bf 1}_N\tilde{x}_k\|^2,                       \label{qq13}
\end{align}
and by (\ref{qq8}) and $\pi_i\in(0,1)$, it has that
\begin{align}
\|h_{2,k}\|^2=\sum_{i=1}^N \|\hat{x}_{i,k}-\tilde{x}_{i,k}\|^2\leq \|x_k-{\bf 1}_N\tilde{x}_k\|^2.                       \label{qq14}
\end{align}

Now, in view of (\ref{P2}), (\ref{qq12}), (\ref{qq13}), (\ref{qq14}) and ${\bf 1}_N\tilde{x}_k=A_\infty x_k$, it has that for $y^*=P_{Fix(F)}(F_\alpha(\tilde{x}_k))\in Fix(F)$,
\begin{align}
&\|\tilde{x}_{k+1}-y^*\|=\|\tilde{x}_k+\alpha\tilde{F}-y^*\|          \nonumber\\
&= \|F_\alpha(\tilde{x}_k)-y^*+\alpha(h_{1,k}+h_{2,k})\|              \nonumber\\
&\leq \|F_\alpha(\tilde{x}_k)-y^*\|+\alpha(\|h_{1,k}\|+\|h_{2,k}\|)       \nonumber\\
&\leq d_{Fix(F)}(F_\alpha(\tilde{x}_k))+\frac{\alpha(\bar{L}+1)}{c_1}\|x_k-A_\infty x_k\|_\pi,         \label{qq15}
\end{align}
where (\ref{eq1}) has been employed in the last inequality and note $F_\alpha:=Id+\alpha(F-Id)$.

To analyze the term $d_{Fix(F)}(F_\alpha(\tilde{x}_k))$ in (\ref{qq15}), invoking Lemma \ref{l1} and (\ref{qq15}) yields that
\begin{align}
&\|\tilde{x}_{k+1}-y^*\|          \nonumber\\
&\leq \rho_3 d_{Fix(F)}(\tilde{x}_k)+\frac{\alpha(\bar{L}+1)}{c_1}\|x_k-A_\infty x_k\|_\pi,         \label{qq15b}
\end{align}

Combining (\ref{qq15}) with $d_{Fix(F)}(\tilde{x}_{k+1})\leq \|\tilde{x}_{k+1}-y^*\|$ can yield that
\begin{align}
d_{Fix(F)}(\tilde{x}_{k+1})&\leq \rho_3 d_{Fix(F)}(\tilde{x}_k)             \nonumber\\
&\hspace{0.4cm}+\frac{\alpha(\bar{L}+1)}{c_1}\|x_k-A_\infty x_k\|_\pi.         \label{qq16}
\end{align}

Finally, by setting $z_k:=col(\|x_k-A_\infty x_k\|_\pi,d_{Fix(F)}(\tilde{x}_k))$ for $k\geq 0$, invoking (\ref{qq11}) and (\ref{qq16}) results in
\begin{align}
z_{k+1}\leq \Theta(\alpha)z_k,          \label{qq17}
\end{align}
where $\Theta(\alpha)$ is defined in (\ref{P8}). Note that $A_\infty x_k={\bf 1}_N\tilde{x}_k$. At this step, Theorem \ref{t3} can be proved by following the similar analysis to that after (\ref{pf17}) for proving Theorem \ref{t1}.

\end{document}